\documentclass[12pt,a4paper]{amsart}
\usepackage[marginratio=1:1]{geometry}

\usepackage[T1]{fontenc}
\usepackage[utf8]{inputenc}
\usepackage[american]{babel}
\usepackage[draft=false]{hyperref}
\usepackage{enumitem}
\usepackage{graphicx}
\usepackage{verbatim}

\usepackage{xcolor}
\hypersetup{
	colorlinks,
	linkcolor={red!50!black},
	citecolor={blue!50!black},
	urlcolor={blue!80!black}
}

\usepackage{amsfonts}
\usepackage{amssymb}
\usepackage{float}
\makeatletter
\providecommand*{\twoheadrightarrowfill@}{%
	\arrowfill@\relbar\relbar\twoheadrightarrow
}
\providecommand*{\twoheadleftarrowfill@}{%
	\arrowfill@\twoheadleftarrow\relbar\relbar
}
\providecommand*{\xtwoheadrightarrow}[2][]{%
	\ext@arrow 0579\twoheadrightarrowfill@{#1}{#2}%
}
\providecommand*{\xtwoheadleftarrow}[2][]{%
	\ext@arrow 5097\twoheadleftarrowfill@{#1}{#2}%
}
\makeatother

\DeclareMathOperator{\tp}{{tp}}

\DeclareMathOperator{\cl}{{cl}}

\DeclareMathOperator{\id}{{id}}
\DeclareMathOperator{\aut}{{Aut}}

\DeclareMathOperator{\dcl}{{dcl}}

\DeclareMathOperator{\age}{{Age}}
\DeclareMathOperator{\im}{{Im}}

\newtheorem{thm}{Theorem}[section]

\newtheorem{lem}[thm]{Lemma}
\newtheorem{fct}[thm]{Fact}

\newtheorem{prop}[thm]{Proposition}

\theoremstyle{remark}
\newtheorem{rem}[thm]{Remark}
\theoremstyle{definition}
\newtheorem{dfn}[thm]{Definition}

\newtheorem{clm*}{Claim}

\newcounter{claimcounter}[thm]
\newenvironment{clm}{\stepcounter{claimcounter}{\noindent {\textbf{Claim}} \theclaimcounter:}}{}

\title[On the topological dynamics of automorphism groups]{On the topological dynamics of automorphism groups; a model-theoretic perspective}
\author{Krzysztof Krupi\'nski}
\email[K.\ Krupi\'{n}ski]{kkrup@math.uni.wroc.pl}
\address[K.\ Krupi\'nski]{
	Instytut Matematyczny, Uniwersytet Wrocławski\\
	pl. Grunwaldzki 2/4\\
	50-384 Wrocław, Poland \hspace{5mm}
ORCID: \href{https://orcid.org/0000-0002-2243-4411}{0000-0002-2243-4411}}

\thanks{The first author is supported by National Science Center, Poland, grants 2015/19/B/ST1/01151, 2016/22/E/ST1/00450, and 2018/31/B/ST1/00357}

\author{Anand Pillay}
\email[A.\ Pillay]{apillay@nd.edu}
\address[A.\ Pillay]{Department of Mathematics, University of Notre Dame\\
	255 Hurley Hall\\
	Notre Dame, IN 46556, USA}
\thanks{The second author was supported by NSF grants DMS-1360702,  DMS-1665035, DMS-1760413, and DMS-2054271}



\date{}

\begin{document}
\maketitle

\begin{abstract}
We give a model-theoretic treatment of the fundamental results of
Kechris-Pestov-Todor\v{c}evi\'{c} theory in the more general context of
automorphism groups of not necessarily countable structures.  One of the main
points is a description of the universal ambit as a certain space of types in an
expanded language.  Using this, we recover results of Kechris-Pestov-Todor\v{c}evi\'{c}
\cite{KPT}, Moore \cite{Mo}, Ngyuen Van Th\'{e} \cite{The},  in the context of automorphism
groups of not necessarily countable structures, as well as Zucker \cite{Zu}. 
\end{abstract}

\section{Introduction}
The idea of studying interactions between dynamical properties of the automorphism group of a Fra\"{i}ss\'e structure and combinatorial properties of the underlying Fra\"{i}ss\'{e} class developed in \cite{KPT}  started a whole new research area which joins techniques from topological dynamics, structural Ramsey theory, and descriptive set theory. The main results of \cite{KPT} are the following.
\begin{itemize}
\item  The automorphism group of a locally finite  Fra\"{i}ss\'{e} structure $F$ is extremely amenable if and only if the underlying  Fra\"{i}ss\'{e} class has the embedding Ramsey property (using the terminology from \cite{Zu}; see \cite[Theorem 5.1]{Zu}). 
\item Let $L_0$ and $L = L_0 \cup \{<\}$ be two languages, $F_0$ be a locally finite Fra\"{i}ss\'{e} structure in $L_0$, and let $F$ be its Fra\"{i}ss\'{e} order expansion to $L$ (meaning that the interpretation of $<$ in $F$ is a linear ordering). Then if $\mbox{Age}(F)$ has the Ramsey property and the so-called ordering property with respect to $\mbox{Age}(F_0)$, then the universal minimal $\aut(F_0)$-flow is the closure $\cl(\aut(F_0) \cdot <)$ in the space of linear orderings on $F$ with the natural left action of $\aut(F_0)$. (See \cite[Theorem 7.5]{KPT}.)
\end{itemize}

Later, the second result was generalized in \cite{The} to so-called precompact expansions of $F_0$ by a possibly infinite number of relation symbols (see Fact \ref{fact: universal minimal flow}).

In this paper, we give a natural account for these (and some other important) results using model-theoretic objects. The key point is our description of the right [left] universal ambit for the group $\aut(M)$ of automorphisms of any structure $M$ as a space of types in a very rich (full) language. To get this description, we use the well-known model-theoretic description of the universal ambit of a topological group from \cite{GPP} (see also \cite{KrPi}). With our approach, it is natural to work more generally with arbitrary structures instead of Fra\"{i}ss\'{e} structures. This naturally yields a generalization of the previous approach. However, since every structure can be canonically expanded to an ultrahomogeneous one (see Subsection \ref{subsection: model theory}), our approach yields in fact a generalization from the context of  Fra\"{i}ss\'{e} structures to possibly uncountable  ultrahomogeneous structures. Some generalizations to an uncountable context have already been obtained, e.g. in \cite{Ba} and \cite{Sa}.  

Our point of view is not based on just trying to translate the existing papers into model theory and translate existing proofs, but to find definitions and proofs natural from the model-theoretic perspective. It would not be surprising that there are parallels between our proofs and those in the descriptive set theory literature. Although our results are not far from the known results,  our paper 
%is a starting point and provides some foundational material for the upcoming paper on ``definable'' versions of various notions and connections between Ramsey theory and topological dynamics for first order theories. 
was a starting point and provided some foundational material for paper \cite{KLM} on ``definable'' versions of various notions and connections between Ramsey theory and topological dynamics for first order theories. 
We also hope that our paper will make the whole subject more natural and easier to understand to a wider model theory society. 

All of this belongs to our general project of studying interactions between model theory and the dynamical properties of groups of automorphisms. 
%As mentioned above, in a forthcoming paper of the first author with Junguk Lee and Slavko Moconja, we will study Ramsey properties and degrees in a first order setting (with some special ``definable'' colorings); but here we focus on classical Ramsey theory (with all possible colorings allowed), and we mostly recover some known results. Independently, also Ehud Hrushovski is preparing a paper containing some first order version of Ramsey theory.
As mentioned above, \cite{KLM} (which was written after this paper) studies Ramsey properties and degrees in a first order setting (working with ``definable'' colorings); but here we focus on classical Ramsey theory (with all possible colorings allowed), and we mostly recover some known results. Independently, also Ehud Hrushovski studied some first order version of Ramsey theory in \cite{Hr} (which was also written after our paper).

In Section \ref{section: preliminaries}, we recall the relevant definitions and facts from model theory, topological dynamics, and Ramsey theory. 
In Section \ref{section: universal ambit}, we give our description of the universal ambit of the group of automorphism of any structure as a space of types, and, using it, we recover Zucker's presentation from \cite{Zu}  of the universal ambit as a certain inverse limit: model-theoretically this becomes absolutely natural, as it follows from the presentation of the type space in infinitely many variables as the inverse limit of the restrictions to the finite tuples of variables. In Section \ref{section: extreme amenability}, we reprove the first main theorem from \cite{KPT} recalled above, and in Section \ref{section: amenability} --- an analogous result from \cite{Mo} characterizing amenability of the automorphism group via the so-called convex Ramsey property (everything done in a more general context of arbitrary structures). In Section \ref{section: universal minimal flow}, we reprove the aforementioned result from \cite{The} (also in a more general context) yielding a description of the universal minimal flow. In Section \ref{section: metrizability}, we reprove \cite[Theorem 8.7]{Zu} saying that metrizability of the universal minimal $\aut(M)$-flow is equivalent to $\age(M)$ having finite embedding Ramsey degree, where $M$ is a Fra\"{i}ss\'{e} structure.

\section{Preliminaries}\label{section: preliminaries}

We present here the necessary notions and facts from model theory, topological dynamics, and Ramsey theory.

\subsection{Model theory}\label{subsection: model theory}

A first order structure will usually be denoted by $M$. We say that $M$ is {\em $\kappa$-saturated} if every type over a subset $A$ of $M$ of cardinality smaller than $\kappa$ is realized in $M$; it is {\em strongly $\kappa$-homogeneous} if every elementary map between subsets of $M$ of cardinality smaller than $\kappa$ extends to an automorphism of $M$. Equivalently, strong $\kappa$-homogeneity means that any tuples $\bar a \equiv \bar b$ in $M$ of length less than $\kappa$ lie in the same orbit under $\aut(M)$. A {\em monster model} of a given complete theory $T$ is a $\kappa$-saturated and strongly $\kappa$-homogeneous model for a sufficiently large cardinal $\kappa$ (usually one assumes that $\kappa$ is a strong limit cardinal greater than $|T|$); it is well-know that a monster model always exists. 

An {\em ultrahomogeneous structure} is a structure $M$ in which every isomorphism between any finitely generated substructures extends to an isomorphism of $M$; if the language is relational, then finitely generated substructures are just finite substructures. Equivalently, ultrahomogeneity means that any finite tuples in $M$ with the same quantifier-free type lie in the same orbit under $\aut(M)$. Note that each ultrahomogeneous structure is strongly $\aleph_0$-homogeneous. A {\em Fra\"{i}ss\'{e} structure} is a countable ultrahomogeneous structure.  It is well-know that that the {\em age} of a Fra\"{i}ss\'{e} structure $M$ (i.e. the class $\mbox{Age}(M)$ of all finitely generated structures in the given language which can be embedded into $M$) is a {\em Fra\"{i}ss\'{e} class}, i.e. is non-empty and satisfies: Hereditary Property (HP), Joint embedding Property (JEP), Amalgamation Property (AP), and Denumerability (see \cite[Section 2]{KPT}). Fra\"{i}ss\'{e}'s theorem says that the converse is true: every Fra\"{i}ss\'{e} class has a unique (up to isomorphism) {\em Fra\"{i}ss\'{e} limit}, i.e. a Fra\"{i}ss\'{e} structure whose age is exactly the Fra\"{i}ss\'{e} class in question.

If $M$ is an arbitrary structure, one can always consider its canonical ultrahomogeneous expansion by adding predicates for all the $\aut(M)$-orbits on all finite Cartesian powers of $M$. The automorphism group of this expansion is the same as the original one. In this paper, often one can pass to this expansion without loss of generality, which in the case of countable $M$ means that we can assume that $M$ is a Fra\"{i}ss\'{e} structure.

\subsection{Topological dynamics}

In this paper, compact spaces are Hausdorff by definition.
Let $G$ be a topological group. Recall that a left [right] {\em $G$-flow} is a pair $(G,X)$ where $X$ is a non-empty, compact space on which $G$ acts on the left [resp. on the right] continuously. A {\em $G$-ambit} is a flow $(G,X,x_0)$ with a distinguished point $x_0 \in X$ whose $G$-orbit is dense in $X$. It is well-known that for any topological group $G$ there exists a universal $G$-ambit, i.e. a $G$-ambit which maps homomorphically to any other $G$-ambit; a universal $G$-ambit is clearly unique up to isomorphism, so we can say {\em the universal $G$-ambit}. The existence of the universal $G$-ambit is also easy: up to isomorphism there are at most $\beth_3({|G|})$ $G$-ambits, so we can find a set $\mathcal{A}$ of $G$-ambits which consists of representatives of all the isomorphism classes; then the product of all $G$-ambits from $\mathcal{A}$, with the distinguished point being the net consisting of the distinguished points in the ambits from $\mathcal{A}$, is the universal $G$-ambit. This universal $G$-ambit can also be described as the Samuel compactification of $G$, and there is a well-known construction of this object in general topology (e.g. see \cite[Section 2]{Us}).  In the next subsection, we recall the model-theoretic presentation of the universal $G$-ambit which we will use in this paper.

A {\em subflow} of a given flow $(G,X)$ is a flow of the form $(G,Y)$ for a closed, $G$-invariant subset $Y$ of $X$, where the action of $G$ on $Y$ is the restriction of the action of $G$ on $X$. A {\em minimal flow} is a flow which does not have proper subflows. A {\em universal minimal $G$-flow} is  a minimal $G$-flow which maps homomorphically to any minimal $G$-flow. By Zorn's lemma, each flow has a minimal subflow. It is clear that that any minimal subflow of the universal $G$-ambit is a universal minimal $G$-flow. It turns out that a universal minimal $G$-flow is also unique up to isomorphism, which is less obvious (see \cite[Chapter 8, Theorem 1]{Au}). An important goal of topological dynamics is to understand the universal minimal $G$-flow for a given group $G$. 

Recall that a topological group $G$ is said to be {\em extremely amenable} if in every left [equivalently right] $G$-flow there is a fixed point. Equivalently, this holds for the universal left [right] $G$-ambit. A topological group $G$ is said to be {\em amenable} if on every left [equivalently right] $G$-flow there is a $G$-invariant, Borel probability measure. Equivalently, this holds for the universal left [right] $G$-ambit.

\subsection{Model-theoretic description of the universal $G$-ambit}\label{subsection: universal G-ambit}

The description of the universal left $G$-ambit given below comes from \cite{GPP}. It can also be found in \cite[Fact 2.11]{KrPi}. In fact, this description is a
model-theoretic interpretation of the Samuel compactification, where ``model-theoretic'' refers to passing to a ``nonstandard model'' or
``elementary extension of the ground model''.

Let $G$ be a topological group. 
Treat it as a first order structure $M$ in any language $L$ in which we have a function symbol interpreted as the group law and for every open subset $U$ of $G$ we have a unary relation symbol (also denoted by $U$) interpreted as $U$. More generally, it is enough to work in any structure $M$ in which $G$ is a $\emptyset$-definable group and all open subsets of $G$ are $\emptyset$-definable.
Let $M^* \succ M$ be a monster model, $G^*:=G(M^*)$, and $U^*:=U(M^*)$. The group $\mu$ of {\em infinitesimals} is defined as $\bigcap \{U^*: U\;\, \mbox{an open neighborhood of the neutral element of}\;\, G\}$. Define a relation $\sim$ on $G^*$ by
$$a \sim b  \iff ab^{-1} \in \mu.$$ 
Finally, define $E_\mu$ on $G^*$ by 
$$E_\mu := \; \sim \circ \equiv_M \;=\; \equiv_M \circ \sim,$$
where $\equiv_M$ is the relation of having the same type over $M$.
Then $E_\mu$ is the finest bounded, $M$-type-definable equivalence relation on $G^*$ coarsening $\sim$. Moreover, $\mu$ is normalized by $G$, so 
$$g\cdot(a/E_\mu): = (ga)/E_\mu$$
is a well-defined action of $G$ on $G^*/E_\mu$, and it turns out (see \cite[Fact 2.11]{KrPi}) that $(G,G^*/E_\mu, e/E_\mu)$ is exactly the universal left $G$-ambit, where $G^*/E_\mu$ is equipped with the {\em logic topology} (i.e. the closed subsets of $G^*/E_\mu$ are those subsets whose preimages under the quotient map are type-definable subsets of $G^*$).

From this, it is easy to get an analogous description of the universal right $G$-ambit. It is clear that it will be $(G,G^*/E_\mu,e/E_\mu)$ with the right action of $G$ on $G^*/E_\mu$ given by $(a/E_\mu) * g:= g^{-1} \cdot (a/E_\mu)= (g^{-1}a)/E_\mu$. 
Now, applying the group-theoretic inverse to everything, we get the relation 
$$E_\mu^r:=  E_\mu^{-1} = \; \sim^r \circ \equiv_M \; = \; \equiv_M \circ \sim^r,$$ 
where $a \sim^r b \iff a^{-1}b \in \mu$, the right action of $G$ on $G^*/E_\mu^r$ given by 
$$(a/E_\mu^r) g: = ((a^{-1}/E_\mu)* g)^{-1} =  ((g^{-1}a^{-1})/E_\mu)^{-1} = (ag)/E_\mu^r,$$ and the universal right $G$-ambit is exactly $(G,G^*/E_\mu^r,e/E_\mu^r)$ with this action.

\subsection{Structural Ramsey theory}\label{subsection: Ramsey theory}

In this paper, we will be talking about colorings of embeddings rather than of substructures (as in \cite{Zu}; in particular, see \cite[Proposition 4.4]{Zu}). Let $\mathcal{C}$ be a class of finite structures in a language $L$. For two finite $L$-structures  $A$ and $B$, by $\mbox{Emb}(A,B)$ we denote the set of all embeddings from $A$ to $B$; $A \leq B$ means that $\mbox{Emb}(A,B) \ne \emptyset$. We say that $\mathcal{C}$ has the {\em embedding Ramsey property} (ERP), if for every $A, B \in \mathcal{C}$ with $A \leq B$ and for any $r \in \omega$ there is $C \in \mathcal{C}$ with $B \leq C$ such that for any coloring $c \colon \mbox{Emb}(A,C) \to r$ there is $f \in \mbox{Emb}(B,C)$ such that $f \circ \mbox{Emb}(A,B)$ is monochromatic with respect to $c$.

%%%%%%%%%%%%%%%%%%%%%%%%%%%%%%%%%%%%%%%%

Now, we recall one of the fundamental results of Kechris, Pestov, Todor\v{c}evi\'{c} theory (see Theorem 5.1 in \cite{Zu}), which we will reprove and generalize in Section \ref{section: extreme amenability}.

\begin{fct}\label{fact: extreme amenability}
If $\mathcal{K}$ is a Fra\"{i}ss\'{e} class of finite structures with Fra\"{i}ss\'{e} limit $K$, then $\mathcal{K}$ has the ERP if and only if $\aut(K)$ is extremely amenable.
\end{fct}

To recall the second main result describing universal minimal flows of automorphism groups of some Fra\"{i}ss\'{e} structures, we need to recall several notions. We will work in the more general context from \cite{The}, with precompact relational expansions in place of expansions by one symbol $<$.

Consider two countable languages $L$ and $L_0$, where $L$ is obtained from $L_0$ by adding countably many relation symbols. Let $\mathcal{K}_0$ be a  Fra\"{i}ss\'{e} class in $L_0$ consisting of finite structures. We say that a class $\mathcal{K}$ of $L$-structures is an {\em expansion} of $\mathcal{K}_0$ if each structure in $\mathcal{K}$ is an expansion of a structure from $\mathcal{K}_0$, and conversely, each structure from $\mathcal{K}_0$ has an expansion to a structure in $\mathcal{K}$. Whenever $\mathcal{K}$ is an expansion of $\mathcal{K}_0$, $\mathcal{K}$ is said to have the {\em expansion property} relative to $\mathcal{K}_0$ if for every $A_0 \in \mathcal{K}_0$ there exists $B_0 \in \mathcal{K}_0$ such that for every $A,B \in \mathcal{K}$ with $A \upharpoonright L_0=A_0$ and  $B \upharpoonright L_0 = B_0$ one has $A \leq B$.

Let $K_0$ be the Fra\"{i}ss\'{e} limit of a Fra\"{i}ss\'{e} class $\mathcal{K}_0$ of finite structures. We say that an $L$-expansion $K$ of $K_0$ is {\em precompact} if each structure from $\mathcal{K}_0=\mbox{Age}(K_0)$ has only finitely many expansions to structures in $\mbox{Age}(K)$.

Denote $L\!\setminus \!L_0 =\{ R_i: i \in I\}$.
The set of $L \! \setminus \! L_0$-structures on a universe $K$ can be naturally treated as the compact space $X:=\prod_{i\in I} \{0,1\}^{K^{n_i}}$ with the product topology, where $n_i$ is the arity of $R_i$. Now, if $K_0$ is a Fra\"{i}ss\'{e} structure with an expansion $K$, then $\vec{R}:=K \upharpoonright (L\! \setminus \! L_0)$ is naturally an element of $X$. Moreover, we have a natural left [and right] action of $\aut(K_0)$ on $X$: the left action is just given by the left translations of each relation from $L \! \setminus \! L_0$ on $K$ (i.e. $gR_i:=\{(gx_1,\dots,gx_{n_i}): (x_1,\dots,x_{n_i}) \in R_i\}$), and the right action is the left action by the inverse.

For the next fact see \cite[Theorem 5]{The}. We focus here only on one direction of this theorem, yielding a description of the universal minimal flow. The closure in this paper is taken with respect to the product topology on $X$. In \cite[Theorem 5]{The}, a finer topology is considered (see \cite[Section 2]{The}), which coincides with the product topology when $L \setminus L_0$ is finite.
%(e.g. in the original Theorem 7.5 from \cite{KPT}). 
%In general, by the minimality of the flow in the conclusion of  \cite[Theorem 5]{The}, the version with the finer topology implies the version with the product topology (more precisely, the closure in the conclusion is the same for both topologies).
%%%In general, by the minimality of the flow in the conclusion of  \cite[Theorem 5]{The}, and the fact that $\cl(\aut(K_0) \cdot \vec{R})$ is compact in the finer topology (see \cite[Definition 1]{The}), the version of the theorem with the finer topology is equivalent to the version with the product topology (more precisely, the closure in the conclusion is the same for both topologies).
In general, by the fact that $\cl(\aut(K_0) \cdot \vec{R})$ is compact in the finer topology (see \cite[Proposition 1]{The}), the version of the theorem with the finer topology is equivalent to the version with the product topology (more precisely, the closure in the conclusion is the same for both topologies).

\begin{fct}\label{fact: universal minimal flow}
Let $K_0$ be a locally finite Fra\"{i}ss\'{e} structure, and $K$ be a Fra\"{i}ss\'{e} precompact, relational
expansion of $K_0$. Assume that the class $\mbox{Age}(K)$ has the ERP as well as the expansion property
relative to $\mbox{Age}(K_0)$. Then the $\aut(K_0)$-subflow $\cl(\aut(K_0) \cdot \vec{R})$ of $X$ is the universal minimal left $\aut(K_0)$-flow. Equivalently, the $\aut(K_0)$-subflow $\cl(\vec{R} \cdot \aut(K_0))$ of $X$ is the universal minimal right $\aut(K_0)$-flow.
\end{fct}

The name ``precompact'' is used, because precompactness of an  $L$-expansion $K$ of a Fra\"{i}ss\'{e} structure $K_0$ is equivalent to topological precompactness  of the metric subspace $\aut(K_0) \cdot \vec{R}$ of $X$ equipped with a certain natural metric (see Section 2 of \cite{The}) which induces a finer topology (the one mentioned before Fact \ref{fact: universal minimal flow}) on $X$ than the product one. But in this paper, we will not use this metric at all.

Let $K$ be the Fra\"{i}ss\'{e} limit of a Fra\"{i}ss\'{e} class $\mathcal{K}$ of finite structures.
Zucker \cite{Zu} found a very interesting connection between metrizability of the universal minimal $\aut(K)$-flow and a Ramsey-theoretic property of $\mathcal{K}$.

\begin{dfn}
A class $\mathcal{C}$ of finite structures in some language has {\em separately finite embedding Ramsey degree} if for every $A \in \mathcal{C}$ there is a natural number $k$ such that for every $B \in \mathcal{C}$ with $A \leq B$ and for every $r \in \omega$ there is $C \in \mathcal{C}$ with $B \leq C$ such that for any coloring $c \colon \mbox{Emb}(A,C) \to r$ there is $f \in \mbox{Emb}(B,C)$ such that $c[f \circ \mbox{Emb}(A,B)]$ is of size at most $k$. 
\end{dfn}

We used the word ``separately'' to reflect the fact that $k$ depends on $A$.
%It is easy to see that in this definition we could equivalently color substructures instead of embeddings. 
By \cite[Proposition 4.4]{Zu}, in this definition we could equivalently color substructures instead of embeddings. 
The next result is \cite[Theorem 8.7]{Zu}.

\begin{fct}\label{fact: metrizability}
Let $K$ be a locally finite Fra\"{i}ss\'{e} structure, and $\mathcal{K} = \age(K)$.
Then the following conditions are equivalent.
\begin{enumerate}
\item The universal minimal $\aut(K)$-flow is metrizable.
\item $\mathcal{K}$ has separately finite embedding Ramsey degree.
\end{enumerate}
\end{fct}

\subsection{Structural Ramsey theory in this paper}\label{section: structure Ramsey theory for arbitrary structures}

Here, we extend the definitions from the previous subsection in the form that we will be using in this paper.

Note that if $A$ and $B$ are structures in the same language, and we enumerate $A$ as $\bar a$, then each embedding $f \in \mbox{Emb}(A,B)$ is naturally identified with the tuple $\bar a' =f(\bar a)$ contained in $B$, and in this way the set of embeddings of $A$ to $B$ is the same thing as ${B \choose \bar a}^{qf} :=\{ \bar a' : \bar a' \subseteq B \;\, \mbox{and}\;\, \bar a' \equiv^{qf} \bar a\}$. Hence, if $\mathcal{K}$ is a Fra\"{i}ss\'{e} class of finite structures in a relational language with Fra\"{i}ss\'{e} limit $K$, then $\mathcal{K}$ has the ERP if and only if for any finite tuple $\bar a$ from $K$ and a finite set $B \subseteq K$ containing $\bar a$, for any $r \in \omega$, there is a finite $C \subseteq K$ containing $B$ such that for every coloring $c\colon {C \choose \bar a}^{qf} \to r$ there is an isomorphic copy $B'\subseteq C$ of $B$ such that ${B' \choose \bar a}^{qf}$ is monochromatic with respect to $c$. But by ultrahomogeneity of $K$, for any finite $\bar a$ in $K$ and any $B\subseteq K$, we have  ${B \choose \bar a}^{qf} :=\{ \bar a' : \bar a' \subseteq B \;\, \mbox{and}\;\, \bar a'= f(\bar a)\;\, \mbox{for some}\;\, f \in \aut(K)\}$. This classical situation leads us to the following generalization, which will be used in our results.

Let $M$ be an arbitrary (possible uncountable) structure in an arbitrary language. For any tuple $\bar a$ in $M$ and $B \subseteq M$, by ${B \choose \bar a}$ we will mean the set $\{ \bar a' : \bar a' \subseteq B \;\, \mbox{and}\;\, \bar a'= f(\bar a)\;\, \mbox{for some}\;\, f \in \aut(M)\}$; an analogous notation applies when $\bar a$ is replaced by a subset $A$ of $M$. A family $\mathcal{A}$ of finite subsets of $M$ is said to be {\em cofinal (in $M$)}  if every finite subset of $M$ is contained in a member of $\mathcal{A}$. 

\begin{dfn}\label{definition: Ramsey property}
\begin{enumerate}
\item We will say that $M$ has the {\em embedding Ramsey property} (ERP) if for any finite tuple $\bar a$ in $M$ (possibly with repetitions) and a finite set $B\subseteq M$ containing $\bar a$, for any $r \in \omega$, there is a finite $C \subseteq M$ containing $B$ such that for every coloring $c\colon {C \choose \bar a} \to r$ there is $B' \in {C \choose B}$ such that ${B' \choose \bar a}$ is monochromatic with respect to $c$.
\item A cofinal family $\mathcal{A}$ of finite subsets of $M$ has the {\em embedding Ramsey property} (ERP) if for any tuple $\bar a$ enumerating a member of $\mathcal A$ and a finite set $B \in \mathcal{A}$ containing $\bar a$, for any $r \in \omega$, there is a finite $C \in \mathcal{A}$ containing $B$ such that for every coloring $c\colon {C \choose \bar a} \to r$ there is $B' \in {C \choose B}$ such that ${B' \choose \bar a}$ is monochromatic with respect to $c$.
\end{enumerate}
\end{dfn}

\begin{rem}\label{remark: equivalence of various definitions of ERP}
The following conditions are equivalent for an arbitrary structure $M$.
\begin{enumerate}
\item $M$ has the ERP.
\item Every structure $M'$ with the same universe as $M$ and the same group of automorphisms has the ERP.
\item Some cofinal family of finite subsets of $M$ has the ERP.
\item Every cofinal family of finite subsets of $M$ has the ERP. 
\end{enumerate}
\end{rem}

\begin{proof}
The implications $(1) \rightarrow (2) \rightarrow (4) \rightarrow (3)$ are trivial.\\
$(3) \rightarrow (1)$. Suppose $\mathcal{A}$ is a cofinal family of finite subsets of $M$ with the ERP. Consider any finite tuple $\bar \alpha$ in $M$. To show the ERP for $\bar \alpha$, we can clearly assume that $\bar \alpha$ does not have repetitions. As $\mathcal{A}$ is cofinal, $\bar \alpha$ can be extended to a tuple $\bar a$ enumerating a member of $\mathcal{A}$. Consider any finite $B \subseteq M$ containing $\bar \alpha$. As $\mathcal{A}$ is cofinal, we can find $\tilde{B} \in \mathcal{A}$ containing $B$ and $\bar a$, and such that every $\bar \alpha' \in {B \choose \bar \alpha}$ extends to an $\bar a' \in {\tilde{B} \choose \bar a}$.   Consider any $r \in \omega$. Since $\mathcal{A}$ has the ERP, we can find $C \in \mathcal{A}$ containing $\tilde{B}$ such that for every coloring $\tilde{c}\colon {C \choose \bar a} \to r$ there is $\tilde{B}' \in {C \choose \tilde{B}}$ such that ${\tilde{B}' \choose \bar a}$ is monochromatic. Now, consider any coloring $c\colon {C \choose \bar \alpha} \to r$. It extends to a coloring $\tilde{c} \colon {C \choose \bar a} \to r$ by giving the color $c$ of the subtuple corresponding to $\bar \alpha$. Since $\tilde{B}' = \sigma[\tilde{B}]$ for some $\sigma \in \aut(M)$, it is clear that $B' := \sigma[B]$ has the desired property that ${B' \choose \bar \alpha}$ is monochromatic with respect to $c$.
\end{proof}

We will return to this general context in Section \ref{section: extreme amenability}. A similar discussion applies to finite embedding Ramsey degrees.

\begin{dfn}\label{definition: finite Ramsey degree of M}
\begin{enumerate}
\item We will say that $M$ has {\em separately finite embedding Ramsey degree} if for any finite tuple $\bar a$ (possibly with repetitions) there exists $k_{\bar a} \in \omega$ such that for every finite $B \subseteq M$ containing $\bar a$ and for any $r \in \omega$ there is a finite $C\subseteq M$ containing $B$ such that for every coloring $c\colon {C \choose \bar a} \to r$ there is $B' \in {C \choose B}$ such that the set $c[{B' \choose \bar a}]$ is of size at most $k_{\bar a}$.
\item A cofinal family $\mathcal{A}$ of finite subsets of $M$ has {\em separately finite embedding Ramsey degree} if for any finite tuple $\bar a$ enumerating a member of $\mathcal A$ there exists $k_{\bar a} \in \omega$ such that for every $B \in \mathcal{A}$ containing $\bar a$ and for any $r \in \omega$ there is a finite $C \in \mathcal{A}$ containing $B$ such that for every coloring $c\colon {C \choose \bar a} \to r$ there is $B' \in {C \choose B}$ such that the set $c[{B' \choose \bar a}]$ is of size at most $k_{\bar a}$.
\end{enumerate}
\end{dfn}

\begin{rem}
The following conditions are equivalent for an arbitrary structure $M$.
\begin{enumerate}
\item $M$ has separately finite embedding Ramsey degree.
\item Every structure $M'$ with the same universe as $M$ and the same group of automorphisms has separately finite embedding Ramsey degree.
\item Some cofinal family of finite subsets of $M$ has separately finite embedding Ramsey degree.
\item Every cofinal family of finite subsets of $M$ has separately finite embedding Ramsey degree. 
\end{enumerate}
\end{rem}

A similar discussion applies to the so-called convex embedding Ramsey property, but this will be handled in Section \ref{section: amenability}.

Note that in all these situations, without loss of generality one can pass to the canonical expansion of $M$ to an ultrahomogeneous structure. So, in fact, our generalizations will be only to uncountable structures (as ultrahomogeneity can be always  assumed without loss of generality).

As a corollary of the above discussions we get that if $M$ is a locally finite Fra\"{i}ss\'{e} structure, then $M$ has the ERP [resp. convex embedding Ramsey property, or separately finite embedding Ramsey degree] if and only if $\mbox{Age}(M)$ has it.

\section{Model-theoretic description of the universal ambit}\label{section: universal ambit}

In this section, $M$ is an arbitrary first order structure in a language $L$, and $G:=\aut(M)$ is equipped with the pointwise convergence topology. We will give a model-theoretic realization of the universal $G$-ambit.

Let $\mathcal{M}$ be the structure consisting of two disjoint sorts $G$ and $M$ with predicates for all the subsets of all the finite Cartesian products of sorts; we call this language {\em full}. Note that the natural action of $G$ on $M$ is $\emptyset$-definable in $\mathcal{M}$, all elements of $\mathcal{M}$ are in $\dcl(\emptyset)$, and all the $L$-definable subsets of the Cartesian powers of $M$ are $\emptyset$-definable in $\mathcal{M}$. Hence, $L$-formulas can naturally be identified with equivalent formulas from the full language. (If one prefers, to our full language one can add all the symbols from $L$.) Types in this full language will be denoted by $\tp^{\textrm{full}}$ and in the original  language $L$ by $\tp^L$. Let $\mathcal{M}^*=(G^*,M^*,\dots) \succ \mathcal{M}$ be a monster model (of the theory of $\mathcal{M}$). Then $G^*$ acts definably and faithfully as a group of automorphisms of $M^*$ treated as an $L$-structure. Enumerate $M$ as $\bar m$. Define

$$\Sigma^{\mathcal{M}} := \{\tp^{\textrm{full}} (\sigma(\bar m)): \sigma \in G^*\} = \{\tp^{\textrm{full}} (\sigma(\bar m)/\mathcal{M}): \sigma \in G^*\} .$$

%By $|L|$ we will mean the cardinality of the set of all formulas in $L$.
\begin{rem}
Let $S^{\mathcal{M}} := \{ p \in S^\textrm{full}(\emptyset): \tp^L(\bar m) \subseteq p\}$. Then:
\begin{enumerate}
\item $\Sigma^{\mathcal{M}}$ is a closed subset of $S^{\mathcal{M}}$.
\item %If $M$ is $\aleph_0$-saturated and strongly $\aleph_0$-homogeneous, then $\Sigma^{\mathcal{M}}=S^{\mathcal{M}}$. In particular, this holds if $M$ is the unique countable model of an $\omega$-categorical theory.
If $M$ is a strongly $\aleph_0$-homogeneous model of an $\omega$-categorical theory  (e.g. the unique countable model), then $\Sigma^{\mathcal{M}}=S^{\mathcal{M}}$. 
\end{enumerate}
\end{rem}

\begin{proof}
(1) follows from $|M|^+$-saturation of $\mathcal{M}^*$, as we get that $\Sigma^{\mathcal{M}}$ is the closed subset of $S^{\mathcal{M}}$ given by the partial type $(\exists \sigma \in G)(\bar x =\sigma(\bar m))$. For (2) consider any $p \in \Sigma^{\mathcal{M}}$. By $\omega$-categoricity, for any finite tuples of variables $\bar x$ and $\bar y$ of the same length, the condition $\tp^L(\bar x) = \tp^L(\bar y)$ is definable by a formula $\varphi(\bar x, \bar y)$ in the language $L$. By the strong $\aleph_0$-homogeneity of $M$,
$$\mathcal{M} \models \varphi(\bar x, \bar y) \rightarrow (\exists \sigma \in G)(\sigma(\bar x) = \bar y),$$
so the same sentence holds in $\mathcal {M}^*$. Now, take any $\bar m' \equiv^L \bar m$ in $M^*$. We conclude that for every corresponding finite subtuples $\bar a'$ and $\bar a$ of $\bar m'$ and $\bar m$, respectively, there is $\sigma \in G^*$ with $\sigma(\bar a)=\bar a'$. By $|M|^+$-saturation of $\mathcal{M}^*$, we get some $\sigma \in G^*$ with 
$\sigma(\bar m)=\bar m'$.
%$\sigma(\bar m) \models \tp^{\textrm{full}}(\bar m')$.
\end{proof}

The main result of this section is the following  description of the universal $G$-ambit.

\begin{thm}\label{theorem: universal G-ambit}
The formula $\tp^{\textrm{full}}(\sigma(\bar m)) \cdot g := \tp^{\textrm{full}}(\sigma(g(\bar m)))$ yields a well-defined right action of $G$ on $\Sigma^{\mathcal{M}}$, and with this action $(G,\Sigma^{\mathcal{M}}, \tp^{\textrm{full}}(\bar m))$ is the universal right $G$-ambit. In particular, this universal ambit is zero-dimensional.
\end{thm}

\begin{proof}
First, we check that $\cdot$ is well-defined.   Suppose $\tp^{\textrm{full}}(\sigma(\bar m)) =  \tp^{\textrm{full}}(\tau(\bar m))$ and $g \in G$ (where $\sigma,\tau \in G^*$). We need to show that  $\tp^{\textrm{full}}(\sigma(g(\bar m))) =  \tp^{\textrm{full}}(\tau(g(\bar m)))$. By strong $|M|^+$-homogeneity of $\mathcal{M}^*$, there is $f \in \aut(\mathcal{M}^*)$ such that $f(\sigma(\bar m)) = \tau(\bar m)$. Since $f$ fixes $\mathcal{M}$ pointwise and the action of $G$ on $M$ is $\emptyset$-definable in the full language, we get $f(\sigma)(\bar m)=f(\sigma)(f(\bar m))=f(\sigma(\bar m))= \tau(\bar m)$. Hence, $f(\sigma(g(\bar m))) = f(\sigma)(f(g)(f(\bar m))) = f(\sigma)(g(\bar m)) = \tau(g (\bar m))$, where the last equality follows from the previous sentence, as $g(\bar m)$ is a permutation of $\bar m$.

The fact that $\cdot$ is a right action is trivial. Next, let us check that $\cdot$ is continuous. 
Consider a basic clopen subset of $\Sigma^{\mathcal{M}}$, i.e. a subset of the form $[\varphi(\bar x)]:=\{ p \in \Sigma^{\mathcal{M}} : \varphi(\bar x) \in p\}$ for some formula $\varphi(\bar x)$ without parameters in the full language. The goal is to show that the set 
$$X:= \{ (q,g) \in \Sigma^{\mathcal{M}} \times G : \varphi(\bar x) \in q\cdot g\}$$
is open in the product topology. Although the tuple of variables $\bar x$ is infinite (corresponding to $\bar m$), the formula $\varphi(\bar x)$ uses only a finite subtuple $\bar x'$ of $\bar x$ corresponding to some finite subtuple $\bar a$ of $\bar m$.
%so from now on, let $\bar x$ be the relevant finite subtuple. 
Note that $\varphi(\bar x) \in q\cdot g$ if and only if there is $\sigma \in G^*$ such that $\varphi(\sigma(g(\bar a)))$ and $q=\tp^{\textrm{full}}(\sigma(\bar m))$. Hence, for any $q \in \Sigma^{\mathcal{M}}$ and $g \in G$, we see that $(q,g) \in X$ if and only if there is $\bar b$ in $M$ such that $g(\bar a) = \bar b$ and the formula $\psi_{\bar b}(\bar x, \bar b) := (\exists \sigma \textrm{ of sort } G)(\varphi(\sigma(\bar b)) \wedge \sigma(\bar b) =\bar x_{\bar b})$ belongs to $q$, where $\bar x_{\bar b}$ is the finite subtuple of $\bar x$ corresponding to the subtuple $\bar b$ of $\bar m$ (recall that $\bar b$ is in $\dcl(\emptyset)$ in the full language, so we can use it as parameters). Therefore, 
$$X= \bigcup_{\bar b \in G\bar a} [\psi_{\bar b}(\bar x, \bar b)] \times \{g \in G: g(\bar a)=\bar b\},$$ 
which is clearly open.

Note that $\tp^{\textrm{full}}(\bar m) \cdot G$ is dense in $\Sigma^{\mathcal{M}}$, as for any $\sigma \in G^*$ and $\varphi(\bar x) \in \tp^{\textrm{full}}(\sigma(\bar m))$, since $\mathcal{M} \prec \mathcal{M}^*$, we get that there is $g \in G$ with $\varphi(g(\bar m))$, but this means that $\varphi(\bar x) \in \tp^{\textrm{full}}(g(\bar m))= \tp^{\textrm{full}}(\bar m) \cdot g$.

So we have already proved that $(G,\Sigma^{\mathcal{M}},  \tp^{\textrm{full}}(\bar m))$ is a right $G$-ambit. To see that it is universal, it is enough to show that it is isomorphic to the universal right  $G$-ambit described as $G^*/E_\mu^r$ in Subsection \ref{subsection: universal G-ambit}.

Let $F \colon G^* \to \Sigma^{\mathcal{M}}$ be given by $F(\sigma) := \tp^{\textrm{full}}(\sigma(\bar m))$.\\

\begin{clm}
$F(\sigma)=F(\tau) \Longleftrightarrow \sigma E_\mu^r \tau$.
\end{clm}

\begin{proof}
$(\Rightarrow)$ Assume $F(\sigma)=F(\tau)$. Then there is $f \in \aut(\mathcal{M}^*)$ with $f(\sigma(\bar m))= \tau(\bar m)$. Since $f(\sigma(\bar m)) = f(\sigma)(f(\bar m))=f(\sigma)(\bar m)$, we conclude that $(\tau^{-1}f(\sigma))(\bar m)) = \bar m$. Since $\mu=\{\sigma \in G^*: \sigma(\bar m)=\bar m\}$, we obtain $\tau \sim^r f(\sigma) \equiv^{\textrm{full}}_{\mathcal{M}} \sigma$ which means that $\tau E_\mu^r \sigma$.\\
$(\Leftarrow)$ Assume $\sigma E_\mu^r \tau$. Then there is $f \in \aut(\mathcal{M}^*)$ with $f(\sigma) \sim^r \tau$. This exactly means that $f(\sigma)(\bar m) = \tau(\bar m)$, so $f(\sigma(\bar m)) = \tau(\bar m)$, and hence $\tp^{\textrm{full}}(\sigma(\bar m)) = \tp^{\textrm{full}}(\tau(\bar m))$.
\end{proof}

By the claim, $F$ induces a bijection $\tilde{F} \colon G^*/E_\mu^r \to \Sigma^{\mathcal{M}}$. $\tilde{F}$ is continuous, as $\tilde{F}^{-1}[[\varphi(\bar x)]] = \{ \sigma/E_\mu^r: \models \varphi(\sigma(\bar m))\}$ is closed in the logic topology (for any formula $\varphi(\bar x)$ without parameters in the full language). Moreover, for any $\sigma \in G^*$ and $g \in G$, $\tilde{F}((\sigma/E_\mu^r)  g)= \tilde{F}((\sigma g)/E_\mu^r) = \tp^{\textrm{full}}(\sigma(g(\bar m))) =  \tp^{\textrm{full}}(\sigma(\bar m)) \cdot g$. Also, $\tilde{F}(\id/E_\mu^r) = \tp^{\textrm{full}}(\bar m)$.

We have justified that $\tilde{F}$ is an isomorphism of right $G$-ambits.
\end{proof}

\begin{rem}
One could significantly shorten the above proof. Namely, everything follows from Claim 1 and the computations following it: $\cdot$ is well-defined, because the action of $G$ on $G^*/E_\mu^r$ is well-defined and $\tilde{F}$ maps the action of $G$ on $G^*/E_\mu^r$ to $\cdot$, and the fact that $(G, \Sigma^{\mathcal{M}},  \tp^{\textrm{full}}(\bar m))$ is a right $G$-ambit follows from the fact that $(G,G^*/E_\mu^r,\id/E_\mu^r)$ is and the observation that $\tilde{F}$ is a homeomorphism preserving the actions of $G$ and mapping $\id/E_\mu^r$ to $\tp^{\textrm{full}}(\bar m)$. 
%Nevertheless, we decided to include a direct proof of the fact that $(G, \Sigma^{\mathcal{M}},  \tp^{\textrm{full}}(\bar m))$ is a right $G$-ambit in order to show what is really going on here.
\end{rem}

Nevertheless, we decided to include a direct proof of the fact that $(G, \Sigma^{\mathcal{M}},  \tp^{\textrm{full}}(\bar m))$ is a right $G$-ambit in order to show what is really going on here, and also because of the following remark (whose context generalizes the one from Theorem \ref{theorem: universal G-ambit}) which follows by almost the same (direct) proof. 

\begin{rem}\label{remark: ambit for other languages}
Let $\mathcal{M}'$ be a structure $(G,M,\dots)$ in a language $L'$ such that the action of $G$ on $M$ 
%and all the $L$-definable subsets of the finite Cartesian powers of $M$ are 
is $\emptyset$-definable in $\mathcal{M}'$. Let $\mathcal{M}'^* = (G^*,M^*,\cdot) \succ \mathcal{M}'$ be a monster model. Put $\Sigma^{\mathcal{M}'}:= \{\tp^{L'}(\sigma(\bar m)/M)): \sigma \in G^*\}$. Then $(G,\Sigma^{\mathcal{M}'}, \tp^{L'}(\bar m/M))$ is a right  $G$-ambit with the right action defined by $\tp^{L'}(\sigma(\bar m)/M) \cdot g := \tp^{L'}(\sigma(g(\bar m))/M)$.
\end{rem}

Since in this remark we work only over parameters from $M$ (and not from all of $\mathcal{M}'$), the computation in the proof of the fact that $\cdot$ is well-defined must be modified as follows: $f(\sigma(g(\bar m))) = f(\sigma)(f(g)(f(\bar m))) = f(\sigma)(f(g(\bar m))) = f(\sigma)(g(\bar m))=\tau(g (\bar m))$, as $g(\bar m)$ is contained in $M$ and $f$ fixes $M$ pointwise.

As an immediate corollary of Theorem \ref{theorem: universal G-ambit}, we get that the universal left $G$-ambit is also $(G,\Sigma^{\mathcal{M}},\tp^{\textrm{full}}(\bar m))$ with the left action given by 
$$g \cdot \tp^{\textrm{full}}(\sigma(\bar m)) :=  \tp^{\textrm{full}}(\sigma(g^{-1}(\bar m))).$$

An important aspect of \cite{Zu} was a presentation of the universal right $G$-ambit (working with a Fra\"{i}ss\'{e} structure) as a certain inverse limit. Here, we will see that this is exactly the obvious presentation of the type space in infinitely many variables as the inverse limit of type spaces in finitely many variables.

For any finite tuple $\bar a$ in $M$, put
$$\Sigma^{\mathcal{M}}_{\bar a}: = \{ \tp^{\textrm{full}}(\sigma(\bar a)): \sigma \in G^*\}.$$ 
In contrast with Theorem \ref{theorem: universal G-ambit} and the last remark, there is no obvious structure  of a right $G$-ambit on $\Sigma^{\mathcal{M}}_{\bar a}$. However, the following remark is clear.

\begin{rem}\label{remark: presentation as an inverse limit}
The restriction maps yield a homeomorphism $h \colon \Sigma^{\mathcal{M}} \to \underset{\bar a}{\varprojlim} \Sigma^{\mathcal{M}}_{\bar a}$. More generally, $\bar a$ can range over any given enumerations of the sets from a given cofinal family of finite subsets of $M$. 
\end{rem}

Via this homeomorphism we induce the structure of a right $G$-ambit on $\underset{\bar a}{\varprojlim} \Sigma^{\mathcal{M}}_{\bar a}$:
$$\langle \tp^{\textrm{full}}(\sigma(\bar a))\rangle_{\bar a} \cdot g := \langle \tp^{\textrm{full}}(\sigma(g(\bar a)))\rangle_{\bar a}.$$

In order to see that this is exactly the presentation from \cite[Section 6]{Zu}, we have to identify $\Sigma^{\mathcal{M}}_{\bar a}$ with some Stone-\v{C}ech compactifications considered in \cite{Zu}. 

For a finite $\bar a$ in $M$, let $A_{\bar a}$ be the orbit of $\bar a$ under $\aut(M)$. The proof of the next result is left as an exercise.

\begin{prop}\label{proposition: our inverse limit is the same as Zucker's}
\begin{enumerate}
\item $\beta A_{\bar a}$ is homeomorphic to $\Sigma^{\mathcal{M}}_{\bar a}$ via 
$f_{\bar a}$ given by $$f_{\bar a}(\mathcal{U}):= \{\varphi(\bar x) \textrm{ in the full language}: A_{\bar a} \cap \varphi(M) \in \mathcal{U}\};$$ the inverse map is given by $$f_{\bar a}^{-1}(\tp^{\textrm{full}}(\sigma(\bar a))) = \{ U \subseteq A_{\bar a}: U(x) \in \tp^{\textrm{full}}(\sigma(\bar a))\}.$$
\item $\underset{\bar a}{\varprojlim} f_{\bar a}$  is a homeomorphism from $\underset{\bar a}{\varprojlim} \beta A_{\bar a}$ to $\underset{\bar a}{\varprojlim} \Sigma^{\mathcal{M}}_{\bar a}$, where for $\bar a$ being a subtuple of $\bar b$ the bonding map from $\beta A_{\bar b}$ to $\beta A_{\bar a}$ is induced by the restriction map from $A_{\bar b}$ to $A_{\bar a}$. 
\end{enumerate}
\end{prop}

We can now induce a structure of a right $G$-ambit on $\underset{\bar a}{\varprojlim} \beta A_{\bar a}$  via  $\underset{\bar a}{\varprojlim} f_{\bar a}$.

Recall that the Stone-\v{C}ech compactification $\beta G$ of $G$ treated as a discrete group is also the right universal $G$-ambit for $G$ treated as a discrete group, with the right action of $G$ on $\beta G$ given by right translation and with the distinguished point being the principal ultrafilter $[e]:= \{ U \subseteq G: e \in U\}$ (see pages 118-119 in \cite{Au}).
Let $\tilde{f} \colon \beta G \to \Sigma^{\mathcal{M}}$ be the unique continuous extension of the map $f \colon G \to \Sigma^{\mathcal{M}}$ given by $f(g):= \tp^{\textrm{full}}(g(\bar m))$; this is exactly the unique epimorphism of right $G$-ambits (for $G$ treated as a discrete group) from $(G, \beta G, [e])$ to $(G, \Sigma^{\mathcal{M}}, \tp^{\textrm{full}}(\bar m))$. (One can check that for any basic clopen subset of $\beta G$ of the form $[U]:=\{ \mathcal{U} \in \beta G : U \in \mathcal{U}\}$ (where $U \subseteq G$), $\tilde{f}[[U]] = \{ \tp^{\textrm{full}}(\sigma(\bar m)): \sigma \in U^*\}$, but we will not use it.)

%\begin{rem}
%\begin{enumerate}
%\item For any basic clopen subset of $\beta G$ of the form $[U]:=\{ V \subseteq G: U \subseteq V\}$ (where $U \subseteq G$), $\tilde{f}[[U]] = \{ \tp^{\textrm{full}}(\sigma(\bar m)): \sigma \in U^*\}$.
%\item $\tilde{f}$ is a homomorphism of right $G$-ambits.
%\end{enumerate}
%\end{rem}

By the above comments and remarks, we have the following sequence of empimorphisms of right $G$-ambits (the first ambit is for $G$ treated as a discrete group), where $j:= (\underset{\bar a}{\varprojlim} f_{\bar a})^{-1}$.

\begin{equation}
			\beta G \xtwoheadrightarrow{\tilde{f}}{} \Sigma^{\mathcal{M}} \xtwoheadrightarrow{h}{} \underset{\bar a}{\varprojlim} \Sigma^{\mathcal{M}}_{\bar a} \xtwoheadrightarrow{j}{} \underset{\bar a}{\varprojlim} \beta A_{\bar a}.
		\end{equation}
So we see that the right $G$-action on $\underset{\bar a}{\varprojlim} \beta A_{\bar a}$ induced via  $j$ is also induced from the right $G$-action on $\beta G$ via $j \circ h \circ \tilde{f}$.

For $g \in G$ and $[g]:=\{U \subseteq G: g \in U\}$ we have:
\begin{equation}
			[g] \xrightarrow{\tilde{f}}{} \tp^{\textrm{full}}(g(\bar m)) \xrightarrow{h}{} \langle \tp^{\textrm{full}}(g(\bar a)) \rangle_{\bar a} \xrightarrow{j}{}  \langle [g(\bar a)] \rangle_{\bar a},
\end{equation}
so $j \circ h \circ \tilde{f} \colon \beta G \to  \underset{\bar a}{\varprojlim} \beta A_{\bar a}$ is the unique continuous map extending the map $G \to  \underset{\bar a}{\varprojlim} \beta A_{\bar a}$ given by $[g] \mapsto  \langle [g(\bar a)] \rangle_{\bar a}$.

Now, if $M$ is ultrahomogeneous, then for any finite $\bar a$ from $M$ the elements of the orbit $A_{\bar a}$ are exactly all the tuples in $M$ with the same qf-type as $\bar a$, so they can be identified with the embeddings of $\bar a$ into $M$. And so, in the case of Fra\"{i}ss\'{e} structures,  the presentation of the universal right $G$-ambit as $\underset{\bar a}{\varprojlim} \beta A_{\bar a}$ coincides with the presentation from \cite[Section 6]{Zu}. But here we do not assume that $M$ is countable and instead of all initial finite subtuples of $M$ we can range over any cofinal family of finite subtuples. Also, ultrahomogeneity is not needed, but we can always assume  it anyway it by considering the canonical expansion of $M$ mentioned at the end of Subsection \ref{subsection: model theory}. 

The point of the above discussion is that model-theoretically the presentation of $\Sigma^{\mathcal{M}}$ as $\underset{\bar a}{\varprojlim} \Sigma^{\mathcal{M}}_{\bar a}$ is straightforward, and we will use it in Section \ref{section: metrizability}
 to give a rather quick proof of metrizability theorem from \cite[Section 8]{Zu}.

\section{Extreme amenability}\label{section: extreme amenability}

In this section, we give a quick proof of Fact \ref{fact: extreme amenability}, based on our description of the universal right $G$-ambit from Theorem \ref{theorem: universal G-ambit}. In fact, our proof works more generally for automorphism groups of arbitrary (possibly uncountable) structures. We take the notation and terminology from Subsection \ref{section: structure Ramsey theory for arbitrary structures} and Section \ref{section: universal ambit}.

%The following remark is folklore.
\begin{rem}\label{remark: C can be replaced by M}
A structure $M$ has the ERP if and only if for any finite tuple $\bar a$ from $M$ and a finite set $B \subseteq M$ containing $\bar a$, for any $r \in \omega$, for every coloring $c\colon {M \choose \bar a} \to r$, there is $B' \in {M \choose B}$ such that ${B' \choose \bar a}$ is monochromatic with respect to $c$.
\end{rem}

\begin{proof}
$(\rightarrow)$ is trivial. For the other direction, suppose for a contradiction that for some $\bar a$, $B$ and $r$ as above, for every finite $C \subseteq M$ containing $B$, the set $K_C$ of colorings ${C \choose \bar a} \to r$ such that for no $B' \in {C \choose B}$ the set ${B' \choose \bar a}$ is monochromatic is non-empty. Clearly each $K_C$ is finite, and for $C \subseteq C'$  there is a map $K_{C'} \to K_C$ induced by the restriction of the domains of the colorings. So we get a non-empty, profinite space $\underset{C}{\varprojlim} K_C$. Take $\eta \in \underset{C}{\varprojlim} K_C$, and define
$c \colon {M \choose  \bar a} \to r$ by 
$$c(\bar a'): = \eta(C)(\bar a')$$ 
for any finite $C \subseteq M$ containing $B$ and $\bar a'$. It is clear that $c$ is well-defined. We claim that $c$ contradicts the right hand side of the remark, more precisely for no $B' \in {M \choose  B}$ the set ${B' \choose \bar a}$ is monochromatic with respect to $c$. Indeed, for such a $B'$ we can take a finite  superset $C \subseteq M$ of $B \cup B'$, and the conclusion follows from the fact that $c \!\upharpoonright \!{C \choose \bar a} = \eta(C) \in K_C$.
\end{proof}

\begin{thm}\label{theorem: extreme amenability}
Let $M$ be an arbitrary first order structure. The following conditions are equivalent.
\begin{enumerate}
\item $G:=\aut(M)$ is extremely amenable as a topological group.
\item $M$ has the ERP.
\end{enumerate}
\end{thm}

\begin{proof}
 The following conditions are equivalent.
\begin{enumerate}
\item[(i)] $G$ is extremely amenable.
\item[(ii)] There is $p \in \Sigma^{\mathcal{M}}$ such that $p \cdot G = \{p\}$.
\item[(iii)] There is $\sigma \in G^*$ such that for every $g \in G$ one has $\tp^{\textrm{full}}(\sigma(g(\bar m))) = \tp^{\textrm{full}}(\sigma(\bar m))$.
\item[(iv)] For every finite tuple $\bar a$ from $M$, for all natural numbers $n,r$, for every $g_0,\dots,g_{n-1}\in G$, for all formulas $\varphi_0(\bar x), \dots, \varphi_{r-1}(\bar x)$ of the full language (with $\bar x$ corresponding to $\bar a$), there exists $\sigma \in G$ such that
$$\bigwedge_{i <r} \bigwedge_{j < n} (\varphi_i(\sigma(g_j(\bar a))) \leftrightarrow \varphi_i(\sigma(\bar a))).$$
\end{enumerate}

The equivalence of (i) and (ii) follows from Theorem \ref{theorem: universal G-ambit}; (ii) $ \leftrightarrow$ (iii) is trivial; the equivalence of (iii) and (iv) follows from $|M|^+$-saturation of $\mathcal{M}$.

$(2) \rightarrow (1)$. We will show that (iv) holds.  So take data as in (iv). Consider the coloring $c \colon {M \choose \bar a} \to 2^r$ given by
$$c(\bar a')(i) := \left\{
\begin{array}{ll}
1 & \textrm{if} \models \varphi_i(\bar a')\\
0 & \textrm{if} \models \neg\varphi_i(\bar a')
\end{array}
\right.
$$
for  $i \in \{ 0,\dots, r-1\}$. Choose a finite $B \subseteq M$ so that $\bar a, g_0(\bar a),\dots,g_{n-1}(\bar a)$ are all contained in $B$. By the ERP, there is $B' \in {M \choose B}$ such that ${B' \choose \bar a}$ is monochromatic. But this implies that for $\sigma \in G$ such that $\sigma[B] = B'$ the conclusion of (iv) holds (because the tuples $\sigma(\bar a), \sigma(g_0(\bar a)),\dots, \sigma(g_{n-1}(\bar a))$ all belong to ${B' \choose \bar a}$).

$(1) \rightarrow (2)$. Consider any finite tuple $\bar a$ from $M$ and any finite $B \subseteq M$ containing $\bar a$. Let $c \colon {M \choose \bar a} \to r$ be a coloring. The fibers of this coloring are subsets of $M^{|\bar a|}$, so they are defined by formulas (in fact predicates) $\varphi_0(\bar x), \dots, \varphi_{r-1}(\bar x)$ of the full language. Let $g_0(\bar a),\dots,g_{n-1}(\bar a)$ be all elements of ${B \choose \bar a}$ (where the $g_i$'s are from $G$). Using (iv), we get $\sigma \in G$ for which ${\sigma[B] \choose \bar a}$ is monochromatic. So we are done by Remark \ref{remark: C can be replaced by M}.
\end{proof}

\section{Amenability}\label{section: amenability}

In this section, we will reprove Moore's theorem \cite{Mo} characterizing amenability of the groups of automorphisms of Fra\"{i}ss\'e structures via the convex Ramsey property of the underlying Fra\"{i}ss\'e classes. We extend the context to arbitrary structures, and also notice that the condition $\epsilon>0$ from the definition of the convex Ramsey property can be replaced by $\epsilon =0$.

There are several equivalent definitions of a Fra\"{i}ss\'e class to have the convex Ramsey property (which we prefer to call the embedding convex Ramsey property), e.g. items (1)-(5) in \cite[Theorem 7.1]{Mo}.
%We choose one of them and generalize it to cofinal families of finite subsets of an arbitrary structure.
We choose one of them (more precisely, our choice is in between the equivalent conditions from items (2) and (5) of \cite[Theorem 7.1]{Mo}, so it is equivalent to them) and generalize it to cofinal families of finite subsets of an arbitrary structure.

If $B \subseteq C$ are subsets of a structure $M$ and $\bar b$ is an enumeration of $B$, by $\left\langle \begin{smallmatrix} C \\ \bar b \end{smallmatrix} \right\rangle$ we denote the affine combinations of copies of $\bar b$ in $C$, i.e. the set of all $\lambda_1 \bar b_1 + \cdots + \lambda_k \bar b_k$, where $k \in \omega \setminus \{0\}$, $\bar b_i \in { C \choose \bar b}$, and $\lambda _1,\dots, \lambda_k \in [0,1]$ with $\lambda_1 + \dots + \lambda_k =1$. When $\bar a$ is a tuple in $B$, $\bar a' \in {B \choose \bar a}$, and $v=\lambda_1 \bar b_1 + \cdots + \lambda_k \bar b_k \in \left\langle \begin{smallmatrix} C \\ \bar b \end{smallmatrix} \right\rangle$, then $v \circ \bar a':= \lambda_1 \bar a_1' + \dots + \lambda_k \bar a_k' \in \left\langle \begin{smallmatrix} C \\ \bar a \end{smallmatrix} \right\rangle$, where for $\sigma_i \in \aut(M)$ with $\sigma_i(\bar b) = \bar b_i$ we put $\bar a_i' = \sigma_i(\bar a')$. Finally, if $c \colon  {C \choose \bar a} \to 2^r = \{0,1\}^r$, then we define $c(v \circ \bar a'):= \lambda_1c(\bar a_1') + \dots + \lambda_kc(\bar a_k') \in [0,1]^r$.

\begin{dfn}\label{definition: ECRP}
\begin{enumerate}
\item A structure $M$ has the {\em embedding convex Ramsey property} (ECRP) if for every $\epsilon >0$, for any finite tuple $\bar a$ from $M$ and any enumeration $\bar b$ of a finite set $B\subseteq M$ containing $\bar a$, for any $r \in \omega$, there is a finite $C \subseteq M$ containing $B$ such that for every coloring $c \colon {C \choose \bar a} \to 2^r$ there is $v \in \left\langle \begin{smallmatrix} C \\ \bar b \end{smallmatrix} \right\rangle$ such that for every $\bar a', \bar a'' \in {B \choose \bar a}$ one has $|c(v \circ \bar a')- c(v \circ \bar a'')|_{\textrm{sup}} \leq \epsilon$, where $|x|_{\textrm{sup}}$ is the supremum norm on $[0,1]^r$. 

Let us say that $M$ has the {\em strong ECRP} if the definition holds with $\epsilon =0$.
\item A cofinal family $\mathcal{A}$ of finite subsets of a structure $M$ has the {\em embedding convex Ramsey property} (ECRP) if for every $\epsilon >0$, for any $\bar a$ enumerating a member of $\mathcal{A}$ and any $\bar b$ enumerating a member $B \supseteq A$ of $\mathcal{A}$, for any $r \in \omega$, there is $C \in \mathcal{A}$ containing $B$ such that for every coloring $c \colon {C \choose \bar a} \to 2^r$ there is $v \in \left\langle \begin{smallmatrix} C \\ \bar b \end{smallmatrix} \right\rangle$ such that for every $\bar a', \bar a'' \in {B \choose \bar a}$ one has $|c(v \circ \bar a')- c(v \circ \bar a'')|_{\textrm{sup}} \leq \epsilon$

Let us say that $\mathcal{A}$ has the {\em strong ECRP} if the definition holds with $\epsilon =0$.
\end{enumerate}
\end{dfn}

In Definition \ref{definition: ECRP}, it is equivalent to consider $\bar a$ with repetitions allowed or without.
Arguing as in the proof of Remark \ref{remark: equivalence of various definitions of ERP}, one gets that Remark \ref{remark: equivalence of various definitions of ERP} holds with ERP replaced by the [strong] ECRP. 

By the same argument as in Remark \ref{remark: C can be replaced by M}, we get

\begin{rem}\label{remark: M in place of C for amenability}
In Definition \ref{definition: ECRP}, the part ``there is a finite $C \subseteq M$ containing $B$'' [``there is $C \in \mathcal{A}$ containing $B$''] can be removed and then, in the rest of the statement, $C$ should be replaced by $M$.
\end{rem}

\begin{thm}
Let $M$ be an arbitrary first order structure. The following conditions are equivalent.
\begin{enumerate}
\item $G:=\aut(M)$ is amenable as a topological group.
\item $M$ has the ECRP.
\item  $M$ has the strong ECRP.
\end{enumerate}
\end{thm}

\begin{comment}%%%%%%%%%%%%%%%%%%%%%%%%%%%%%%%%%%%%%%
\begin{thm}
Let $M$ be an arbitrary first order structure. The following conditions are equivalent.
\begin{enumerate}
\item $G:=\aut(M)$ is amenable as a topological group.
\item Some cofinal family of finite subsets of $M$ has the ECRP.
\item Every cofinal family of finite subsets of $M$ has the ECRP.
\item Some cofinal family of finite subsets of $M$ has the strong ECRP.
\item Every cofinal family of finite subsets of $M$ has the strong ECRP.
\end{enumerate}
\end{thm}
\end{comment}%%%%%%%%%%%%%%%%%%%%%%%%%%%%%%%%%%%%%%%%%

\begin{proof}
%Consider any cofinal family $\mathcal{A}$ of finite subsets of $M$. \\

\begin{clm}
The following conditions are equivalent.
\begin{enumerate}
\item[(i)] $G$ is amenable.
\item[(ii)] For every $\epsilon >0$, for every finite tuple $\bar a$ from $M$ without repetitions, for all $n,r\in \omega$, for all $g_0,\dots,g_{n-1}\in G$, for all formulas $\varphi_0(\bar x), \dots, \varphi_{r-1}(\bar x)$ of the full language (with $\bar x$ corresponding to $\bar a$), there exist $\lambda_1,\dots,\lambda_k \in [0,1]$ with $\lambda_1 + \dots + \lambda_k =1$ and $h_1,\dots, h_k \in G$ such that for 
$$\mu := \lambda_1 \tp^{\textrm{full}}(h_1(\bar m)) + \dots + \lambda_k \tp^{\textrm{full}}(h_k(\bar m))$$
for all $j_1,j_2 <n$ and $i<r$ one has $|\mu([\varphi_i(\bar x)] \cdot g_{j_1}^{-1}) - \mu([\varphi_i(\bar x)] \cdot g_{j_2}^{-1})| \leq \epsilon$, where $[\varphi_i(\bar x)]$ is the basic clopen set in $\Sigma^{\mathcal{M}}$ consisting of the types containing $\varphi_i(\bar x)$.
\item[(iii)] The same as in (ii) but with $\epsilon =0$.
\end{enumerate}
\end{clm}

\begin{proof}
(iii) $\rightarrow$ (ii) is trivial. To see (ii) $\rightarrow$ (i), it is enough to recall that regular Borel probability measures on a zero-dimensional compact space (such as $\Sigma^{\mathcal{M}}$) are the same thing as finitely additive, probability measures on the Boolean algebra of clopen subsets (e.g. see \cite[416Q(a)]{Fr}),
%\cite[Section 7.1]{Si}), 
and the set of all such measures with the topology inherited from the product $[0,1]^{\textrm{clopens}}$ is a compact space. Then the intersection of the closed sets of measures satisfying the inequalities in (ii) will be non-empty, and any measure in the intersection of these sets will be invariant.

(i) $\rightarrow$ (iii). Take any invariant, Borel probability measure $\nu$ on $\Sigma^{\mathcal{M}}$, and consider any data as in the assumptions of (iii). In particular, we have formulas $\varphi_0(\bar x), \dots, \varphi_{r-1}(\bar x)$ and elements $g_0,\dots,g_{n-1} \in G$. Let $F_0,\dots, F_{nr-1}$ be all the clopens $[\varphi_i(\bar x)] \cdot g_j^{-1}$ for $i < r$ and $j<n$.  We claim that we can find $\lambda_1,\dots,\lambda_k \in [0,1]$ with $\lambda_1 + \dots + \lambda_k =1$ and $h_1,\dots, h_k \in G$ such that for $\mu:= \lambda_1 \tp^{\textrm{full}}(h_1(\bar m)) + \dots + \lambda_k \tp^{\textrm{full}}(h_k(\bar m))$ for all $j < nr$ we have $\mu(F_j) = \nu(F_j)$. This will clearly imply (iii), as $\nu$ is $G$-invariant. In order to show the existence of $\lambda_s$ and $h_s$, consider the atoms $B_1,\dots,B_{k}$ of the Boolean algebra generated by $F_0,\dots,F_{nr-1}$. By the density in $\Sigma^{\mathcal{M}}$ of the types realized in $\mathcal{M}$, we can find $h_s \in G$ for $s\leq k$ such that $\tp^{\textrm{full}}(h_s(\bar m)) \in B_s$. Then put $\lambda_s:= \nu(B_s)$.
\end{proof}

\begin{clm}
Consider any finite tuple $\bar a$ from $M$ without repetitions, natural numbers $n,r,k$, elements $g_0,\dots,g_{n-1}\in G$, formulas $\varphi_0(\bar x), \dots, \varphi_{r-1}(\bar x)$ of the full language (with $\bar x$ corresponding to $\bar a$), $\lambda_1,\dots,\lambda_k \in [0,1]$, elements $h_1,\dots, h_k \in G$, and 
$$\mu := \lambda_1 \tp^{\textrm{full}}(h_1(\bar m)) + \dots + \lambda_k \tp^{\textrm{full}}(h_k(\bar m)).$$
Let $c \colon {M \choose \bar a} \to 2^r$ be given by $c(\bar a')(i)  := \left\{
\begin{array}{ll}
1 & \textrm{if} \models \varphi_i(\bar a')\\
0 & \textrm{if} \models \neg\varphi_i(\bar a')
\end{array}.
\right.$
Let $v := \lambda_1 h_1(\bar b)+ \dots +\lambda_k h_k(\bar b)$ for some $\bar b$ enumerating a finite $B \subseteq M$ containing $\bar a, g_0\bar a,\dots, g_{n-1}\bar a$.
Then for any $j_1,j_2 <n$: 
$$|c(v \circ (g_{j_1}(\bar a)))- c(v \circ (g_{j_2}(\bar a)))|_{\textrm{sup}}= \sup_i |\mu([\varphi_i(\bar x)] \cdot g_{j_1}^{-1}) - \mu([\varphi_i(\bar x)] \cdot g_{j_2}^{-1})|.$$
\end{clm}
\begin{proof}
$
|c(v \circ (g_{j_1}(\bar a)))- c(v \circ (g_{j_2}(\bar a)))|_{\textrm{sup}} = |c(\sum \lambda_s h_s(g_{j_1}(\bar a))) - c(\sum \lambda_s h_s(g_{j_2}(\bar a)))|_{\textrm{sup}}= |\sum \lambda_s c(h_s(g_{j_1}(\bar a))) - \sum \lambda_s c(h_s(g_{j_2}(\bar a)))|_{\textrm{sup}} =\sup_i |\sum\{ \lambda_s: \varphi_i(\bar x) \in \tp^{\textrm{full}}(h_s(g_{j_1}(\bar a)))\} - \sum\{ \lambda_s: \varphi_i(\bar x) \in \tp^{\textrm{full}}(h_s(g_{j_2}(\bar a)))\} |=  \sup_i |\mu([\varphi_i(\bar x)] \cdot g_{j_1}^{-1}) - \mu([\varphi_i(\bar x)] \cdot g_{j_2}^{-1})|.$
\end{proof}

Now, we turn to the implications between (1)-(3).
The implication (3) $\rightarrow$ (2) is trivial. 

(2) $\rightarrow$ (1). By Claim 1, it is enough to show that (ii) holds. But this follows from Claim 2 and the ECRP of $M$.

(1) $\rightarrow$ (3). Consider any finite $\bar a$ from $M$ without repetitions, and a finite $B \subseteq M$ (enumerated as $\bar b$) containing $\bar a$. Let $c \colon {M \choose \bar a} \to 2^r$ be a coloring. Then there are formulas (in fact predicates) $\varphi_0(\bar x), \dots, \varphi_{r-1}(\bar x)$ of the full language such that $c(\bar a')(i)  := \left\{
\begin{array}{ll}
1 & \textrm{if} \models \varphi_i(\bar a')\\
0 & \textrm{if} \models \neg\varphi_i(\bar a')
\end{array}.
\right.$
Let $g_0(\bar a),\dots,g_{n-1}(\bar a)$ be all elements of ${B \choose \bar a}$ (where the $g_i$'s are from $G$). Now, take $\lambda_1,\dots,\lambda_k \in [0,1]$ and $h_1,\dots, h_k \in G$ provided by Claim 1(iii). Then, by Claim 2,  $v := \lambda_1 h_1(\bar b)+ \dots +\lambda_k h_k(\bar b)$ witnesses the strong ECRP (by Remark \ref{remark: M in place of C for amenability}).
\end{proof}

\section{Universal minimal flow}\label{section: universal minimal flow}

In this section, we reprove Fact \ref{fact: universal minimal flow} in the more general setting. Throughout, $L_0 \subseteq L$ are two first order languages. Whenever $M$ is an $L$-structure, its reduct to $L_0$ will be denoted by $M_0$. Then clearly $\aut(M) \leq \aut(M_0)$.

\begin{dfn}
Let $M$ be a structure in $L$, and let $\mathcal{A}$ be a cofinal family of finite subsets of $M$. We say that $\mathcal{A}$ has:
\begin{enumerate}
\item the {\em right expansion property for $(L_0,L)$} if for every $A \in \mathcal{A}$ there is $B \in \mathcal{A}$ such that for every $\sigma \in \aut(M_0)$ there is $\tau \in \aut(M)$ with $\tau[A] \subseteq \sigma[B]$,
\item the {\em left expansion property for $(L_0,L)$} if for every $A \in \mathcal{A}$ there is $B \in \mathcal{A}$ such that for every $\sigma \in \aut(M_0)$ there is $\tau \in \aut(M)$ with $\tau[\sigma[A]] \subseteq B$,
\item the {\em expansion property for $(L_0,L)$} if for every $A \in \mathcal{A}$ there is $B \in \mathcal{A}$ such that for every $\sigma_1, \sigma_2 \in \aut(M_0)$ there is $\tau \in \aut(M)$ with $\tau[\sigma_1[A]] \subseteq \sigma_2[B]$.
\end{enumerate}
\end{dfn}

%It is clear that if $M$ is ultrahomogeneous, we can equivalently speak of embeddings in place of automorphisms.

\begin{dfn}
We will say that a structure $M$ in $L$ has the {\em [right or left] expansion property for $L_0$} if the family of all finite subsets of $M$ has the [resp. right or left] expansion property for $(L_0,L)$.
\end{dfn}

It is clear that $M$ has the [right or left] expansion property for $L_0$ if and only if some (equivalently every) cofinal family of finite subsets of $M$ has the [resp. right or left] expansion property for $(L_0,L)$.

The next remark is left as an exercise. (Recall from Subsection \ref{subsection: Ramsey theory} that for finite structures $A$ and $B$ in the same language, $A \leq B$ means that $\mbox{Emb}(A,B) \ne \emptyset$.)

\begin{comment}%%%%%%%%%%%%%%%%%%%%%%%%%%%%%%%%%%%%%
\begin{rem} Assume $M$ and $M_0=M \! \upharpoonright \! L_0$ are Fra\"{i}ss\'e structures in relational languages $L \supseteq L_0$, $\mathcal{K} := \age(M)$ and $\mathcal{K}_0 := \age(M_0)$. Then:
\begin{enumerate}
\item $\mathcal{K}_0$ has the expansion property for $(L_0,L)$ if and only if $\mathcal{K}$ has the expansion property relative to $\mathcal{K}_0$ in the sense of van Th\'{e} (see Subsection \ref{subsection: Ramsey theory});
\item  $\mathcal{K}_0$ has the right expansion property for $(L_0,L)$ if and only if for every $A \in \mathcal{K}$ there exists $B_0 \in \mathcal{K}_0$ such that for every $B \in \mathcal{K}$ with $B\! \upharpoonright \! L_0 = B_0$ one has $A \leq B$;
\item   $\mathcal{K}_0$ has the left expansion property for $(L_0,L)$ if and only if for every $A_0 \in \mathcal{K}_0$ there exists $B \in \mathcal{K}$ such that for every $A\in \mathcal{K}$ with $A \! \upharpoonright \! L_0=A_0$ one has $A \leq B$.
\end{enumerate}
\end{rem}
\end{comment}%%%%%%%%%%%%%%%%%%%%%%%%%%%%%%%%%%%%
\begin{rem} Assume $M$ and $M_0:=M \! \upharpoonright \! L_0$ are locally finite Fra\"{i}ss\'e structures in languages $L \supseteq L_0$, respectively, where $L \setminus L_0$ consists of relation symbols. Let $\mathcal{K} := \age(M)$ and $\mathcal{K}_0 := \age(M_0)$. Then:
%Assume $M$ and $M_0:=M \! \upharpoonright \! L_0$ are Fra\"{i}ss\'e structures in relational languages $L \supseteq L_0$, $\mathcal{K} := \age(M)$ and $\mathcal{K}_0 := \age(M_0)$. Then:
\begin{enumerate}
\item $M$ has the expansion property for $L_0$ if and only if $\mathcal{K}$ has the expansion property relative to $\mathcal{K}_0$ in the sense of \cite{The} (see Subsection \ref{subsection: Ramsey theory});
\item  $M$ has the right expansion property for $L_0$ if and only if for every $A \in \mathcal{K}$ there exists $B_0 \in \mathcal{K}_0$ such that for every $B \in \mathcal{K}$ with $B\! \upharpoonright \! L_0 = B_0$ one has $A \leq B$;
\item   $M$ has the left expansion property for $L_0$ if and only if for every $A_0 \in \mathcal{K}_0$ there exists $B \in \mathcal{K}$ such that for every $A\in \mathcal{K}$ with $A \! \upharpoonright \! L_0=A_0$ one has $A \leq B$.
\end{enumerate}
\end{rem}

\begin{dfn}
We will say that a structure $M$ in $L$ is {\em precompact} for $L_0$ if each $\aut(M_0)$-orbit on a finite Cartesian power of $M$ is a union of finitely many $\aut(M)$-orbits.
\end{dfn}

Clearly, if $M$ and $M_0$ are ultrahomogeneous, then the above orbits are the same thing as tuples with the same quantifier-free types.

Again, the following easy remark is left as an exercise.

\begin{rem}
 Assume $M$ and $M_0:=M \! \upharpoonright \! L_0$ are locally finite Fra\"{i}ss\'e structures in languages $L \supseteq L_0$, respectively,  where $L \setminus L_0$ consists of relation symbols. Then $M$ is precompact for $L_0$ if and only if $M$ is a precompact expansion of $M_0$ in the sense of \cite{The} (see Subsection \ref{subsection: Ramsey theory}).
% Assume $M$ and $M_0:=M \! \upharpoonright \! L_0$ are Fra\"{i}ss\'e structures in relational languages $L \supseteq L_0$, respectively. Then $M$ is precompact for $L_0$ if and only if $M$ is a precompact expansion of $M_0$ in the sense of Nguyen Van Th\'{e} (Subsection \ref{subsection: Ramsey theory}).
\end{rem}

\begin{rem}
Let $M$ be an $L$-structure which is precompact for $L_0$. Then $M$ has the right expansion property for $L_0$ if and only if it has the expansion property for $L_0$.
\end{rem}

\begin{proof} 
$(\leftarrow)$ is trivial. To show $(\rightarrow)$, consider any finite $A \subseteq M$.
By precompactness, there are $\sigma_1,\dots,\sigma_n \in \aut(M_0)$ such that for every $\sigma \in \aut(M_0)$ there is $\tau \in \aut(M)$ with $\tau[\sigma[A]] = \sigma_i[A]$ for some $i \in \{1,\dots,n\}$. By the right expansion property, we can find finite $B_1,\dots,B_n \subseteq M$ such that for every $i \in \{1,\dots,n\}$,  for every $\sigma \in \aut(M_0)$, there is $\tau \in \aut(M)$ with $\tau[\sigma_i[A]] \subseteq \sigma[B_i]$. Put $B:=B_1 \cup \dots \cup B_n$. We will show that it witnesses the expansion property for $A$. For this, consider any $g_1, g_2 \in \aut(M_0)$. By the choice of the $\sigma_i$'s, there exists $\tau_1 \in \aut(M)$ such that $\tau_1[g_1[A]] = \sigma_i[A]$ for some $i \in \{1,\dots,n\}$. Next, by the choice of $B_i$, there exists $\tau_2 \in \aut(M)$ such that $\tau_2[\sigma_i[A]] \subseteq g_2[B_i] \subseteq g_2[B]$. Hence, $\tau_2 \circ \tau_1 \in \aut(M)$ and $(\tau_2 \circ \tau_1)[g_1[A]] \subseteq g_2[B]$.
\end{proof}

When $M$ is an $L$-structure and $M_0:= M \! \upharpoonright \! L_0$, we have structures $\mathcal{M}$ and $\mathcal{M}_0$ defined as at the beginning of Section \ref{section: universal ambit}. And $\Sigma^{\mathcal{M}}$  is the universal right $\aut(M)$-ambit, while $\Sigma^{\mathcal{M}_0}$ is the universal right $\aut(M_0)$-ambit. Note that $\Sigma^{\mathcal{M}}$ can and will be naturally treated as an $\aut(M)$-subflow of $\Sigma^{\mathcal{M}_0}$.

Now, we turn to the main results.

\begin{thm}\label{theorem: universal minimal flow general result}
Let $M$ be a structure in $L$ with the right expansion property for $L_0$. Assume $\aut(M)$ fixes a point $p = \tp^{\textrm{full}}(\sigma(\bar m))\in \Sigma^{\mathcal{M}}$ (where $\sigma \in \aut(M)^*$), i.e. $p \cdot \aut(M) = \{ p\}$ (by Theorem \ref{theorem: extreme amenability}, this is equivalent to saying that $M$ has the ERP). Then $\cl(p \cdot \aut(M_0))$ is the universal minimal right $\aut(M_0)$-flow.
\end{thm}

\begin{proof}

By universality of the right $\aut(M_0)$-ambit $\Sigma^{\mathcal{M}_0}$, it is enough to show that the subflow $\cl(p \cdot \aut(M_0))$ is minimal.
Suppose for a contradiction that there is $q \in \cl(p \cdot \aut(M_0))$ such that $p \notin \cl(q \cdot \aut(M_0))$. Then we can find $\varphi(\bar x) \in p$ such that

\begin{equation}\label{3}
q = \lim_i p \cdot g_i = \lim_i \tp^{\textrm{full}}(\sigma(g_i(\bar m))) \textrm{ for some net } (g_i)_i \textrm{ in } \aut(M_0),
\end{equation}
and
\begin{equation}\label{4}
\varphi(\bar x) \notin q \cdot h \textrm{ for all } h \in \aut(M_0).
\end{equation}

The formula $\varphi(\bar x)$ uses only a finite subtuple $\bar x'$ of $\bar x$ corresponding to a finite subtuple $\bar a$ of $\bar m$. Let $A$ be the set of all coordinates of $\bar a$. Take a finite $B\subseteq M$ witnessing the right expansion property for $A$. Choose $h_1,\dots,h_n \in \aut(M_0)$ so that $h_1(\bar a), \dots, h_n(\bar a)$ are all the $\aut(M_0)$ -conjugates of $\bar a$ contained in $B$.

By (\ref{3}), (\ref{4}), and the continuity of the action of $\aut(M_0)$ on $\Sigma^{\mathcal{M}_0}$, there is $i$ for which

\begin{equation}\label{5}
\varphi(\bar x) \notin (p \cdot g_i)\cdot h_j \textrm{ for all } j \in \{ 1,\dots,n\}.
\end{equation}

%By the choice of $B$ and $h_1,\dots,h_n$, there exists $\tau \in \aut(M)$ such that $\tau(\bar a) = g_ih_j(\bar a)$ for some $j \in \{ 1,\dots,n\}$. 
By the choice of $B$, there exists $\tau \in \aut(M)$ such that $\tau[A] \subseteq g_i[B]$.  Then $\tau(\bar a) = g_i(\bar a')$ for some tuple $\bar a' \subseteq B$, and so $g_i^{-1}\tau(\bar a) = \bar a'$. By the choice of $h_1,\dots,h_n$, since $g_i^{-1}\tau \in \aut(M_0)$, we get that $\bar a' = h_j(\bar a)$ for some $j \in \{ 1,\dots,n\}$. Thus, $\tau(\bar a) = g_ih_j(\bar a)$.
 
On the other hand, since $p \cdot \aut(M) = \{ p\}$, we have $p= p \cdot \tau = \tp^{\textrm{full}}(\sigma(\tau(\bar m)))$, so $\models \varphi(\sigma(\tau(\bar m)))$, so $\models \varphi(\sigma(\tau(\bar a)))$. 

Therefore, $\models \varphi(\sigma g_ih_j(\bar a))$ which means that $\varphi(\bar x) \in (p \cdot g_i) \cdot h_j$, a contradiction with (\ref{5}).
\end{proof}

Now, we reprove Fact \ref{fact: universal minimal flow}, extending the context to uncountable structures.

\begin{thm}
Let $M$ be an ultrahomogeneous $L$-structure, and let $L \setminus L_0 =\{ R_i: i \in I\}$ consist of relation symbols. Assume $M$ is precompact for $L_0$, has the ERP and the [right] expansion property for $L_0$. Let $\vec{R}:= M \! \upharpoonright \! (L \setminus L_0)$ be an element of the right $\aut(M_0)$-flow $X$ of all  $L \setminus L_0$-structures on the universe of $M$. Then the $\aut(M_0)$-subflow $\cl(\vec{R} \cdot \aut(M_0))$ is the universal minimal right $\aut(M_0)$-flow.
\end{thm}

\begin{proof}
By Theorem \ref{theorem: extreme amenability}, there is $p = \tp^{\textrm{full}}(\sigma(\bar m))\in \Sigma^{\mathcal{M}}$ (where $\sigma \in \aut(M)^*$) with $p \cdot \aut(M) = \{ p\}$. 

Let $\Phi \colon \Sigma^{\mathcal{M}_0} \to X$ be given by declaring that $\Phi(\tp^{\textrm{full}}(\rho(\bar m)))$ (for $\rho \in \aut(M_0)^*$) is a structure $M^\rho$ with the same universe as $M$, where 
$$M^\rho \models R_i(\bar a) \iff M^* \models R_i(\rho(\bar a)).$$

It is clear that $\Phi$ is continuous and preserves the right $\aut(M_0)$-actions.
Therefore, by Theorem \ref{theorem: universal minimal flow general result}, it remains to show that $\Phi \! \upharpoonright \!  \cl(p \cdot \aut(M_0))$ is injective and with the image equal to $\cl(\vec{R} \cdot \aut(M_0))$.

\begin{clm}
\begin{enumerate}
\item For every tuple $\bar a$ in $M$, if $\bar a' \in \aut(M) \bar a$, then $\tp^{\textrm{full}}(\sigma(\bar a))= \tp^{\textrm{full}}(\sigma(\bar a'))$.
\item For every tuple $\bar a$ in $M$ (or even in $M^*$), $\tp^{\textrm{qf}}_L(\sigma (\bar a)) = \tp^{\textrm{qf}}_L(\bar a)$, where $\tp^{\textrm{qf}}_L(\bar \alpha)$ denotes the qf-type of $\bar \alpha$ in $L$.
\item For any finite tuples $\bar a, \bar a'$ in $M$, 
$\tp^{\textrm{qf}}_L(\sigma(\bar a)) = \tp^{\textrm{qf}}_L(\sigma(\bar a'))$ if and only if $\bar a' \in \aut(M) \bar a$.
\end{enumerate}
\end{clm}

\begin{proof}
(1) is clear from the choice of $p$.

(2) follows from the fact that $\sigma \in \aut(M)^*$, and so $\sigma$ acts on $M^*$ as an $L$-automorphism.

(3) follows from (2), finiteness of $\bar a$ and $\bar a'$, and ultrahomogeneity of $M$.
\end{proof}

Note that the assignment $\Phi(\tp^{\textrm{full}}(\rho(\bar m))) \mapsto  \tp^{\textrm{qf}}_L(\rho(\bar m))$ yields a homeomorphic identification of $\im(\Phi)$ with a closed subset of the space of the qf-types in $L$.
%which we will be using without mentioning.

We will show now  that $\Phi \! \upharpoonright \!  \cl(p \cdot \aut(M_0))$ is injective; in other words, each type in $\cl( p \cdot \aut(M_0))$ is determined by the quantifier free $L$-type.
%First, notice that for $\rho_1,\rho_2 \in \aut(M_0)^*$,
%$$\Phi(\tp^{\textrm{full}}(\rho_1(\bar m))) = \Phi(\tp^{\textrm{full}}(\rho_2(\bar m))) \iff \rho_1(\bar m) \equiv^{\textrm{qf}}_L \rho_2(\bar m),$$ 
%which we will be using freely.
%
Take any $q,r \in \cl(p \cdot \aut(M_0))$ such that $\Phi(q)=\Phi(r)$. We have $q=\lim_i \tp^{\textrm{full}}(\sigma(g_i(\bar m)))$ and $r = \lim_j \tp^{\textrm{full}}(\sigma(h_j(\bar m)))$ for some nets $(g_i)_i$ and $(h_j)_j$ from $\aut(M_0)$. By continuity of $\Phi$ and the last paragraph, we get $\lim_i \tp^{\textrm{qf}}_L(\sigma(g_i(\bar m))) = \lim_j \tp^{\textrm{qf}}_L(\sigma(h_j(\bar m)))$.

Take any finite tuple $\bar a$ in $M$. By precompactness and Claim 1(2), there are only finitely many qf-types in $L$ of the elements of $\sigma\aut(M_0) \bar a$. So, by the equality of the above limits, we get that there are some $i_0$ and $j_0$ such that for all $i> i_0$ and $j > j_0$, we have
$\tp^{\textrm{qf}}_L(\sigma(g_i(\bar a)))  = \tp^{\textrm{qf}}_L(\sigma(h_j(\bar a)))$.
By Claim 1(3), this implies that $h_j (\bar a) \in \aut(M) g_i(\bar a)$, and so, by Claim 1(1), $\tp^{\textrm{full}}(\sigma(g_i(\bar a))) =  \tp^{\textrm{full}}(\sigma(h_j(\bar a)))$. Therefore, $\lim_i \tp^{\textrm{full}}(\sigma(g_i(\bar a))) = \lim_j \tp^{\textrm{full}}(\sigma(h_j(\bar a)))$. Since this holds for any finite $\bar a$, we conclude that $q=\lim_i \tp^{\textrm{full}}(\sigma(g_i(\bar m))) = \lim_j \tp^{\textrm{full}}(\sigma(h_j(\bar m))) =r$, so injectivity is proved.

It remains to check that  $\Phi[\cl(p\cdot \aut(M_0))]= \cl(\vec{R} \cdot \aut(M_0))$. 
Since $\Phi[\cl(p \cdot \aut(M_0))]$ is a minimal $\aut(M_0)$-flow (as an image of a minimal $\aut(M_0)$-flow), it is enough to show that $\vec{R} \in \Phi[\cl(p\cdot \aut(M_0))]$. And for that it suffices to check that $\vec{R} = \Phi(p)$. But this is clear by the following equivalences
$$\Phi(p) \models R_i(\bar a) \iff M^* \models R_i(\sigma(\bar a)) \iff M^* \models R_i(\bar a) \iff M \models R_i(\bar a),$$
where the middle one is by Claim 1(2).
\end{proof}

\section{Metrizability of the universal minimal flow}\label{section: metrizability}

We will reprove here  Fact \ref{fact: metrizability}. This time we do not extend the context to uncountable structures.
%but we extend it slightly to arbitrary cofinal families of finite subsets by working in the context of Definition \ref{definition: finite Ramsey degree of M}. 
We will be working in the context of Definition \ref{definition: finite Ramsey degree of M}, using results and notations from Section \ref{section: universal ambit}.

\begin{comment}%%%%%%%%%%%%%%%%%%%%%%%%%%%%%%%%%%%%%%%
As usual, we prefer to talk about colorings of copies of finite tuples rather than embeddings of finite sets (both are cleraly equivalent). More precisely, following Definition \ref{definition: Ramsey property}, we introduce

\begin{dfn}
A cofinal family $\mathcal{A}$ of finite subsets of a structure $M$ has {\em separately finite embedding Ramsey degree} if for any finite tuple $\bar a$ enumerating a member of $\mathcal A$ there exists $k_{\bar a} \in \omega$ such that for every $B \in \mathcal{A}$ containing $\bar a$ and for any $r \in \omega$ there is a finite $C \in \mathcal{A}$ containing $B$ such that for every coloring $c\colon {C \choose \bar a} \to r$ there is $B' \in {C \choose B}$ such that the set $c[{B' \choose \bar a}]$ is of size at most $k_{\bar a}$.
\end{dfn}

\begin{lem}\label{lemma: extension of finite Ramsey deg}
Assume that a cofinal family $\mathcal{A}$ of finite subsets of a structure $M$ has separately finite embedding Ramsey degree. Then for every $A \in \mathcal{A}$, for every $B \in \mathcal{A}$ containing $A$, and for any $r \in \omega$ there is a finite $C \in \mathcal{A}$ containing $B$ such that for every colorings $c_{\bar \alpha} \colon {C \choose \bar \alpha} \to r$, with $\bar \alpha$ ranging over the finite tuples from $A$ enumerating some members of $\mathcal{A}$, there is $B' \in {C \choose B}$ such that each set $c_{\bar \alpha}[{B' \choose \bar \alpha}]$ is of size at most $k_{\bar \alpha}$.
\end{lem}
\end{comment}%%%%%%%%%%%%%%%%%%%%%%%%%%%%%%%%%%%%%%%%%%%

The same argument as in the proof of Remark \ref{remark: C can be replaced by M} yields

\begin{rem}\label{remark: M instead of C for Ramsey degree}
A structure $M$ has separately finite Ramsey degree with witnessing numbers $k_{\bar a}$ if and only if for any finite tuple $\bar a$ in $M$ and a finite set $B \subseteq M$ containing $\bar a$, for any $r \in \omega$, for every coloring $c\colon {M \choose \bar a} \to r$, there is $B' \in {M \choose B}$ such that $|c[{B' \choose \bar a}]| \leq k_{\bar a}$.
\end{rem}

\begin{lem}\label{lemma: extension of finite Ramsey deg}
Assume that a structure $M$ has separately finite embedding Ramsey degree witnessed by numbers $k_{\bar a}$. Then for every finite $A\subseteq M$, for every finite $B \subseteq M$ containing $A$, and for any $r \in \omega$, there is a finite $C\subseteq M$ containing $B$ such that for all colorings $c_{\bar \alpha} \colon {C \choose \bar \alpha} \to r$, with $\bar \alpha$ ranging over the finite tuples from $A$, there is $B' \in {C \choose B}$ such that each set $c_{\bar \alpha}[{B' \choose \bar \alpha}]$ is of size at most $k_{\bar \alpha}$.
\end{lem}

\begin{proof}
By induction on $n$, we will show that for every finite tuples $\bar a_1,\dots,\bar a_n$ from $M$, for every finite $B\subseteq M$ containing all these tuples, and for any $r \in \omega$, there is a finite $C \subseteq M$ containing $B$ such that for all colorings $c_{\bar a_i} \colon {C \choose \bar a_i} \to r$, $i \in \{1,\dots,n\}$, there is $B' \in {C \choose B}$ such that each set $c_{\bar a_i}[{B' \choose \bar a_i}]$ is of size at most $k_{\bar a_i}$. 

The base induction step is obvious by the definition of separately finite Ramsey degree. For the induction step, consider any finite tuples  $\bar a_1,\dots,\bar a_{n+1}$ from $M$ and a finite subset $B$ of $M$ containing these tuples. Let $r \in \omega$. 

By the base induction step, we can find a finite $C_{n+1} \subseteq M$ containing $B$ such that  for every coloring $c \colon {C_{n+1} \choose \bar a_{n+1}} \to r$ there is $B' \in {C_{n+1} \choose B}$ such that $c[{B' \choose \bar a_{n+1}}]$ is of size at most $k_{\bar a_{n+1}}$. By the induction hypothesis applied to $\bar a_1,\dots,\bar a_n$ and to $C_{n+1}$ in place of $B$, we get a finite $C \subseteq M$ containing $C_{n+1}$ such that  for all colorings $c_{\bar a_i} \colon {C \choose \bar a_i} \to r$, $i \in \{1,\dots,n\}$, there is $C_{n+1}' \in {C \choose C_{n+1}}$ such that each set $c_{\bar a_i}[{C_{n+1}' \choose \bar a_i}]$ is of size at most $k_{\bar a_i}$.

Now, consider any colorings $c_{\bar a_i} \colon {C \choose \bar a_i} \to r$, $i \in \{1,\dots,n+1\}$.  Choose $C_{n+1}' \in {C \choose C_{n+1}}$ provided by the last paragraph. 
%Then $C_{n+1}' = f[C_{n+1}]$ for some $f \in \aut(M)$. Put $B' := f[B]$. 
%By the choice of $C_{n+1}$, we easily get that there is $B'' \in {C_{n+1}' \choose B'}$ with $c_{n+1}[{B'' \choose \bar a_{n+1}}]$ of size at most $k_{\bar a_{n+1}}$. Thus, by the choice of $C_{n+1}'$, we conclude that for every $i \in \{1,\dots,n+1\}$, the size of  $c_{\bar a_i}[{B' \choose \bar a_i}]$ is bounded by  $k_{\bar a_i}$. 
Then we easily get that there is $B' \in {C_{n+1}' \choose B}$ with $c_{\bar a_{n+1}}[{B' \choose \bar a_{n+1}}]$ of size at most $k_{\bar a_{n+1}}$. Thus, by the choice of $C_{n+1}'$, we conclude that for every $i \in \{1,\dots,n+1\}$, the size of  $c_{\bar a_i}[{B' \choose \bar a_i}]$ is bounded by  $k_{\bar a_i}$. 
\end{proof}

\begin{thm}\label{theorem: metrizability and Ramsey degree}
Let $M$ be a countable structure, and $G:=\aut(M)$. Then the following conditions are equivalent.
\begin{enumerate}
\item The universal minimal $G$-flow is metrizable.
\item $M$ has separately finite embedding Ramsey degree.
\end{enumerate}
\end{thm}

\begin{proof}
$(2) \rightarrow (1)$. Let the separately finite Ramsey degree be witnessed by the numbers $k_{\bar a}$.  Consider any formulas $\varphi_1(\bar x_1), \dots, \varphi_n(\bar x_n)$ (without parameters) in the full language and any finite $A \subseteq M$. Let $\Delta = \{ \varphi_1(\bar x_1), \dots,\varphi_n(\bar x_n)\}$. For each finite tuple $\bar \alpha \subseteq A$, let $c_{\bar \alpha}\colon {M \choose \bar \alpha} \to 3^n$ be given by 
$$c(\bar \alpha)(i) := \left\{
\begin{array}{ll}
1 & \textrm{if} \models \varphi_i(\bar \alpha')\\
0 & \textrm{if} \models \neg\varphi_i(\bar \alpha')\\
2 & \textrm{if $\bar \alpha'$ is not in the domain of $\varphi_i(\bar x_i)$, i.e. $|\bar \alpha'| \ne |\bar x_i|$}.
\end{array}
\right.
$$

By Lemma \ref{lemma: extension of finite Ramsey deg} applied to $B:=A$, there exists $A'\subseteq M$ and $\sigma_{\bar \varphi, A} \in G$ mapping $A$ to $A'$ such that for every finite tuple $\bar \alpha$ from $A$ and for every $g_1,\dots,g_m \in G$ such that $g_1(\bar \alpha),\dots, g_m(\bar \alpha) \in A$, there are at most $k_{\bar \alpha}$ $\Delta$-types of the tuples  $\sigma_{\bar \varphi, A}(g_1(\bar \alpha)), \dots, \sigma_{\bar \varphi, A}(g_m(\bar \alpha))$. By saturation of $\mathcal{M}^*$, this implies that there is $\sigma \in G^*$ such that for every finite tuple $\bar \alpha$ in $M$ and for every finite set $\Delta'$ of formulas in the full language in variables $\bar x$ corresponding to $\bar \alpha$, one has $|\{ \tp_{\Delta'}(\sigma(g(\bar \alpha))): g \in G\}| \leq k_{\bar \alpha}$. Hence,

\begin{equation}\label{7}
|\{ \tp^{\textrm{full}}(\sigma(g(\bar \alpha))): g \in G\}| \leq k_{\bar \alpha}.
\end{equation}

Remark \ref{remark: presentation as an inverse limit} and the comments afterwards yield an isomorphism $h \colon \Sigma^{\mathcal{M}} \to \underset{\bar a}{\varprojlim} {\Sigma^{\mathcal{M}}_{\bar a}}$ of right $G$-ambits, which satisfies 
$$h[\tp^{\textrm{full}}(\sigma(\bar m)) \cdot G] \subseteq \underset{\bar a}{\varprojlim}  X_{\bar a} \subseteq \underset{\bar a}{\varprojlim} {\Sigma^{\mathcal{M}}_{\bar a}},$$
where $X_{\bar a}: =\{\tp^{\textrm{full}}(\sigma(g(\bar a))): g \in G\} $.
By (\ref{7}), each $X_{\bar a}$ is finite, so the set in the middle is a profinite space, so it is closed in $\underset{\bar a}{\varprojlim} {\Sigma^{\mathcal{M}}_{\bar a}}$. Also, the set on the left is clearly dense in the middle set. Hence, 
$$h[\cl(\tp^{\textrm{full}}(\sigma(\bar m)) \cdot G)] = \underset{\bar a}{\varprojlim}  X_{\bar a}.$$

Since $M$ is countable, there are only countably many finite $\bar a$'s. Since also each $X_{\bar a}$ is finite, we conclude that  $\underset{\bar a}{\varprojlim}  X_{\bar a}$ is second countable and compact and so metrizable (by Urysohn's metrization theorem, see \cite[Theorem 4.2.8 or 4.2.9]{En}), which means that $h[\cl(\tp^{\textrm{full}}(\sigma(\bar m)) \cdot G)]$ is metrizable. This implies that $\cl(\tp^{\textrm{full}}(\sigma(\bar m)) \cdot G)$ is metrizable. But the last flow is a subflow of the universal right $G$-ambit, hence the universal minimal right $G$-flow is a homomorphic image of $\cl(\tp^{\textrm{full}}(\sigma(\bar m)) \cdot G)$, and as such it is also metrizable (again by  Urysohn's metrization theorem, because the image of a second countable, compact space under a continuous map to a Hausdorff space is second countable, which easily follows using networks and \cite[Therorem 3.1.19]{En}).

$(1) \rightarrow (2)$. The universal minimal right $G$-flow  is of the form $\cl(\tp^{\textrm{full}}(\sigma(\bar m)) \cdot G)$ for some $\sigma \in G^*$. Consider any finite $\bar a$ in $M$. Let $\pi_{\bar a}: \Sigma^{\mathcal{M}} \to \Sigma^{\mathcal{M}}_{\bar a}$ be the restriction map. By assumption, $\cl(\tp^{\textrm{full}}(\sigma(\bar m)) \cdot G)$ is metrizable, so $\pi_{\bar a}[\cl(\tp^{\textrm{full}}(\sigma(\bar m)) \cdot G)]$ is also metrizable. On the other hand, by Proposition \ref{proposition: our inverse limit is the same as Zucker's}, $\Sigma^{\mathcal{M}}_{\bar a} \cong \beta A_{\bar a}$. 
%Recall also that $\beta X$ is not metrizable whenever $X$ is an infinite discrete space (if it was, then it would be second countable \cite[Theorem 4.2.8]{En}, so it would contain only countably many clopen subsets, but there are clearly uncountably many basic clopen subsets in $\beta X$). 
Recall that whenever $X$ is a discrete space, then every infinite closed subset of $\beta X$ embeds $\beta \omega$ (see \cite[Cor. 9.12 and Exc. 9.H.2]{GiMe}), and so it is non-metrizable.
Hence, $\pi_{\bar a}[\cl(\tp^{\textrm{full}}(\sigma(\bar m)) \cdot G)]$ is finite, and so $\{ \tp^{\textrm{full}}(\sigma(g(\bar a))): g \in G\}$ is finite; denote its cardinality by $k_{\bar a}$. We check that the $k_{\bar a}$'s witness that $M$ has separately finite embedding Ramsey degree. 

Consider any finite $B \subseteq M$ containing $\bar a$ and a coloring $c \colon {M \choose \bar a} \to r$ for some $r \in \omega$. The fibers of $c$ are defined by formulas (in fact predicates) $\varphi_0(\bar x), \dots, \varphi_{r-1}(\bar x)$ of the full language; put  $\Delta := \{ \varphi_0(\bar x_0), \dots,\varphi_{r-1}(\bar x_{r-1})\}$.  Let $g_0(\bar a),\dots,g_{n-1}(\bar a)$ be all elements of ${B \choose \bar a}$ (where the $g_i$'s are from $G$). By the choice of $k_{\bar a}$, we have $|\{ \tp_\Delta(\sigma(g(\bar a))): g \in G\}| \leq k_{\bar a}$. Hence, there is $h \in G$ with $|\{ \tp_\Delta(h(g_i(\bar a))): i \in n\}| \leq k_{\bar a}$, which means that $|c[{h[B] \choose \bar a}]| \leq k_{\bar a}$. So we are done by Remark \ref{remark: M instead of C for Ramsey degree}.
\end{proof}

Let $M$ be a countable structure and $G:=\aut(M)$.
We finish with another characterization of metrizability of the universal minimal $G$-flow. Remark \ref{remark: ambit for other languages} tells us that for any language $L'$ in which the action of $G$ on $M$ is $\emptyset$-definable, we have a natural structure of a right $G$-ambit on $\Sigma^{\mathcal{M}'}$. For such a language $L'$, by $L''$ we will denote the relational language of the Morleyization restricted to $M$ of the theory of $\mathcal{M}'=(G,M, \dots)$ in the language $L'$ expanded by constants from $M$, i.e. for every $L'$-formula $\varphi(\bar x)$ with parameters from $M$ and with $\bar x$ corresponding to some sorts of $M$, we have a relation symbol $R_\varphi(\bar x)$ in $L''$. Note that if $L'$ is countable, so is $L''$. Let $X$ be the right $G$-flow consisting of all the $L''$-structures with the universe $M$, where everything is defined in a standard way (as in the second paragraph above Fact \ref{fact: universal minimal flow}). In particular, the right action of $G$ on $X$ is given by: $R_\varphi g :=\{(g^{-1}a_1,\dots,g^{-1}a_n): (a_1,\dots,a_n) \in R_{\varphi}\}$.

\begin{rem}\label{remark: the map Phi}
The function $\Phi \colon \Sigma^{\mathcal{M}'} \to X$ decribed by  
$$\Phi(\tp^{L'}(\sigma(\bar m)/M)) \models R_\varphi (\bar \alpha) \iff \mathcal{M}' \models \varphi(\sigma(\bar \alpha))$$
is a monomorphism of right $G$-flows.
\end{rem}

Let $M(G)$ be a universal minimal right $G$-flow contained in $\Sigma^{\mathcal{M}}$.

\begin{prop}\label{proposition: M(G) as inverse limit for countable languages}
The following conditions are equivalent.
\begin{enumerate}
\item $M(G)$ is metrizable.
\item There is a countable language $L'$ as above for which the restriction map from $\Sigma^{\mathcal{M}}$ to $\Sigma^{\mathcal{M}'}$ restricted to $M(G)$ is injective (then clearly the image of $M(G)$ under this map is the universal minimal right $G$-flow).
\item There is a countable language $L'$ and an $L''$-structure $N$ in $X$ such that $\cl(N \cdot G)$ is the universal minimal right $G$-flow.
\end{enumerate}
\end{prop}

\begin{proof}
$(3) \rightarrow (1)$ is obvious, and $(2) \rightarrow (3)$ follows from Remark \ref{remark: the map Phi}. 

$(1) \rightarrow (2)$. Since $M(G)$ is assumed to be metrizable, and we know by Theorem \ref{theorem: universal G-ambit} that it is zero-dimensional, it has a countable basis consisting of clopen sets. These sets are given by formulas in a countable sublanguage $L'$ of the full language. It is clear that such an $L'$ works in (2).
\end{proof}

In the proof of Theorem \ref{theorem: metrizability and Ramsey degree}, the presentation of $\Sigma^{\mathcal{M}}$ as $\underset{\bar a}{\varprojlim} \Sigma^{\mathcal{M}}_{\bar a}$ from Remark \ref{remark: presentation as an inverse limit} was essential. But there is also another natural  presentation, namely
$$\Sigma^{\mathcal{M}} \cong \underset{L'}{\varprojlim} \Sigma^{\mathcal{M}'},$$
where $L'$ ranges over the countable sublanguages of the full language in which the action of $G$ on $M$ is $\emptyset$-definable, and where the isomorphism is given by the restriction maps to the sublanguages.
This clearly induces an isomorphism 
$$M(G) \cong \underset{L'}{\varprojlim} M^{L'}(G),$$
where each $M^{L'}(G)$ is the minimal $G$-subflow of $\Sigma^{\mathcal{M}'}$ obtained from $M(G)$ by the restriction to $L'$.

An obvious corollary of Proposition \ref{proposition: M(G) as inverse limit for countable languages} is that $M(G)$ is metrizable if and only if for some countable language $L'$ (which can be assumed to be a sublanguage of the full language) already the map $M(G) \to M^{L'}(G)$ is an isomorphism of $G$-flows.

\begin{rem}
The following conditions are equivalent.
\begin{enumerate}
\item $M(G)$ is metrizable.
\item Some $G$-subflow $\Sigma$ of $\Sigma^{\mathcal{M}}$ is metrizable.
\item For some $G$-subflow $\Sigma$ of $\Sigma^{\mathcal{M}}$ and some coutnable language $L'$ as above, the restriction map $\Sigma  \to \Sigma^{\mathcal{M}'}$ is injective.
\end{enumerate}
\end{rem}

\begin{proof}
The equivalence of (1) and (2) follows from universality of the ambit $\Sigma^{\mathcal{M}}$. The implication $(3) \rightarrow (2)$ is obvious, and 
%$(2) \rightarrow (3)$ follows by the same argument as in $(1) \rightarrow (2)$ in Proposition \ref{proposition: M(G) as inverse limit for countable languages}.
$(1) \rightarrow (3)$ follows by Proposition \ref{proposition: M(G) as inverse limit for countable languages}.
\end{proof}

\section*{Acknowledgments}
We would like to thank the referee for careful reading and useful suggestions.

Data availability statement: Data sharing not applicable to this article as no datasets were generated or analyzed during the current study.\\

\end{document}